\documentclass[a4paper,12pt]{article}
\usepackage{latexsym}
\usepackage{amsmath}
\usepackage{amssymb}
\usepackage{enumerate}
\usepackage{enumitem}
\usepackage{tikz}
\usepackage{color}

\usepackage{theorem}
\newtheorem{theorem}{Theorem}[section]
\newtheorem{proposition}[theorem]{Proposition}
\newtheorem{lemma}[theorem]{Lemma}

\newtheorem{CMTC}{Corollary to Theorem C}

\newtheorem{MTA}{Theorem A}
\newtheorem{MTB}{Theorem B}
\newtheorem{MTC}{Theorem C}

\theorembodyfont{\rmfamily}
\newtheorem{proof}{\textmd{\textit{Proof.}}}

\newtheorem{remark}[theorem]{Remark}
\newtheorem{example}[theorem]{Example}
\newtheorem{definition}[theorem]{Definition}

\makeatletter

\@addtoreset{equation}{section}
\makeatother

\newcommand{\qedd}{\hfill \Box}
\newcommand{\ve}{\varepsilon}

\newcommand{\R}{\ensuremath{\mathbb{R}}}

\newcommand{\cC}{\ensuremath{\mathcal{C}}}

\title{The Singular Locus of an Almost Distance Function 
\footnote{
Mathematics Subject Classification (2010)\,:\,53C22, 53C60.}
\footnote{
Keywords: almost distance function, Busemann function, copoint, coray (asymptotic ray),  cut
locus,  Finsler surface, Lipschitz version of the inverse function theorem,  1-Lipschitz function,  ray, reverse cut locus, Sard theorem, singular locus.}
}
\author{Minoru TANAKA}
\date{}
\begin{document}


\maketitle

\begin{abstract}
The aim of this article is to generalize the notion of the cut locus and to get the structure theorem for it.
For this purpose,  we first  introduce a class  of  1-Lipschitz  functions on a Finsler manifold, each member of 
which is called an {\it almost distance function}.
Typical examples of an almost distance function are 　the distance function from a point and the Busemann functions.

 The generalized notion of the cut locus in this paper is called 
the  {\it singular locus}  of an almost distance function.
The singular locus  consists of the {\it upper singular locus} and the {\it lower singular locus.}
The upper singular locus coincides with the cut locus of a point $p$
for the distance function from the point $p,$ and the lower singular locus coincides with the set of all copoints of a ray $\gamma$ when the almost distance function is the Busemann function of the ray $\gamma$.
Therefore, it is possible to treat the cut locus 
of a closed subset and the set of copoints of a ray in a unified way　
by introducing the singular locus for the almost  distance function.

In this article, some theorems on the distance function from a closed set and  the Busemann function are generalized by making use of the almost distance function.

\end{abstract}

\section{Introduction}

Since Poincar\'e \cite{P} introduced the notion of the {\it cut locus} of a point in a compact convex surface in 1905, this notion has been generalized for a submanifold of an (arbitrary dimensional) Riemannian manifold,  and a closed subset of a Riemannian manifold or a Finsler manifold.

The cut locus of  a point of a Riemannian manifold is  the most fundamental case when
 we consider  the cut locus.
The cut locus of  a point $p$ in a complete Riemannian manifold $M$ is very closely related to
the distance  function $d_p$ from the point $p.$ 
It is well known that the cut locus equals the closure of the set of all non-differentiable points of the function $d_p$ on $M\setminus\{p\}.$

This  function $d_p$ has two important properties:
One is 1-Lipschitz, i.e.,
for any points $x,y$ in $M,$
$d_p(x)-d_p(y)\le d(y,x).$
Here $d(\cdot,\cdot)$ denotes the Riemannian distance on $M.$
The other one is
that for any unit speed minimal geodesic segment $\gamma :[0,a]\to M$ emanating from $p,$
$d_p(\gamma(t))=t$ holds on $[0,a].$

If the geodesic segment $\gamma$ 
 has no extensions as a minimal geodesic, then the end point $\gamma(a)$ is called a {\it cut point} of $p$ along $\gamma,$ and the set of the cut points along all minimal geodesic segments  issuing from $p$ is called the {\it cut locus} of $p$.

In this paper we  generalize the notion of the cut locus. 
For this purpose,
we first  introduce 
  an {\it almost distance function} on a Finsler manifold.
We need  the  notion of $f$-geodesics for introducing the almost distance function $f.$

Let $f$ be a 1-Lipschitz function on a connected Finsler manifold $M.$
A unit pseed geodesic segment $\alpha :I \to M$ is called an $f$-{\it geodesic} if
$f(\alpha(t))-f(\alpha(s))=t-s$ holds for any $s,t\in I,$ where $I$ denotes an interval.

A 1-Lipschitz function $f$  on a connected Finsler manifold $M$ is called an {\it almost distance function} if for each point $p\in f^{-1}(\inf f,\sup f),$ there exist a neighborhood $U_p$ of $p$ and  a positive constant
$\delta(p)$ such that for each point $q\in U_p,$ there exists an $f$-geodesic $\gamma_q:[0,\delta(p)]\to M$  with $q=\gamma_q(0)$ or $q=\gamma_q(\delta(p)).$

The distance function $d_p$ is an almost distance  function on a connected Finsler manifold.
Another typical example is the Busemann function  $B_\gamma$ of a ray $\gamma$ on a forward complete connected Finsler manifold. 
It is known that 
 a ray $\sigma$ is  a {\it coray} of  a ray $\gamma$ if and only if the ray  $\sigma$ is  a $B_\gamma$-geodesic and that for  any ray $\gamma,$ there exists a coray of $\gamma$ emanating from
each point of the manifold (see for example \cite{N1,Oh,Sa}).

The class of almost distance functions is very rich as   the last example in Section 7 shows that the {\it singular locus } of an almost distance function can have  a very complicated structure, which never happens for the distance function from a point in a Riemannian manifold.
The {\it singular locus} $\cC(f)$ of an almost distance function $f$
is defined as the set of the end points of  all maximal $f$-geodesics.
For example, the singular locus of the distance function $d_p$ from a point $p$ on a bi-complete Finsler manifold equals  the union of the cut locus of $p$ and $\{p\}.$

In this article, we will prove three main theorems, {\bf Theorems A, B,} and  {\bf C}
which have been already proved  in the case where the almost distance function $f$ is the distance function from a closed subset or a Busemann function.

The following theorem generalizes  \cite[Theorem 2.5]{KT1},  \cite[Theorem 1.2]{Sa} and \cite[Theorem 2.3]{ST}.

\begin{MTA}
Let $f$ be an almost distance function on a connected  Finsler manifold $M$ and $p\in f^{-1}(\inf f, \sup f).$  
Then  $f$ is differentiable at $p$ if and only if $p$ admits  a unique $f$-geodesic. 
Moreover, if $f$ is differentiable at a point $p$ in $f^{-1}(\inf f, \sup f)$, its differential 
$df_p$  at $p$ is given by 
\begin{equation}\label{diff f}
df_{p}(v)=g_{\dot\gamma(f(p))}({\dot\gamma(f(p))},v)
\end{equation}
for any $v\in T_pM,$ where $\gamma$ denotes the unique maximal $f$-geodesic through $p=\gamma(f(p))$ with canonical parameter, and $\dot\gamma(f(p))$ denotes the velocity vector of $\gamma$ at $p.$ 
Here, the {\it canonical } parameter of the $f$-geodesic $\alpha$ means  that  the parameter
$t$
satisfies $f\circ\alpha(t)=t.$ 
\end{MTA}

The following theorem generalizes Theorem B in \cite{ST}.

\begin{MTB}
Let $f$ be an almost distance function on a bi-complete  2-dimensional Finsler manifold $M.$
Then, the singular locus $ \cC(f)\cap f^{-1}(\inf f,\sup f)$ in $f^{-1}(\inf f,\sup f)$ of $f$ satisfies the following properties.
\begin{enumerate}
\item
The set $\cC(f)\cap f^{-1}(\inf f,\sup f)$ is a local tree and any  two points in the same connected component can be joined by a rectifiable curve in $\cC(f)\cap f^{-1}(\inf f,\sup f).$
\item
The topology of $\cC(f)\cap f^{-1}(\inf f,\sup f)$ induced from the intrinsic metric $\delta$
 coincides with the induced topology of $\cC(f)\cap f^{-1}(\inf f,\sup f)$ from $M$. 
\item
The set $\cC(f)\cap f^{-1}(\inf f,\sup f)$ is a union of countably many rectifiable Jordan arcs except for the end points of $\cC(f)\cap f^{-1}(\inf f,\sup f).$
\end{enumerate}
\end{MTB}
\begin{remark}
 Nasu \cite{N2}  investigates the topological structure of the singular locus of a Busemann function on a finitely connected 2-dimensional Finsler manifold (more generally a Busemann G-space)  of nonpositive curvature.
 K. Shiohama and the present author \cite{ShT} proved Theorem B for the distance function from a compact subset of an Alexandrov surface and for the Busemann function on a noncompact  Alexandrov surface. 
\end{remark}

The following theorem is called {\it Sard theorem for an almost distance function}.
Here  a  point is called  a {\it critical point} (in the sense of Clarke)  of an almost distance function $f$ if
 the {\it generalized differential }of $f$ at $p$ contains  the zero 1-form.
The detailed definition is given in Section 6.

For  the distance function $d_N$ from a 
 closed smooth  submanifold $N$ of an  n-dimensional   complete Riemannian manifold,  it is proved that the set of 
all  critical values of $d_N$ is of measure zero, i.e., J. Itoh and 
the present author  \cite{IT} proved it  when $n<5$  and  L. Rifford  \cite{R}  generalized it  for any $n.$

\begin{MTC}
Let $f$ be an almost distance function on a bi-complete  2-dimensional Finsler manifold $M.$
Then,  the set $C_V(f)$ of critical values of $f$ is of measure zero.
\end{MTC}

\begin{remark}
Theorem C is still true for an arbitrary dimensional Riemannian manifold if we add  some  assumptions to the function $f$ (see Theorem \ref{remThC}).
\end{remark}

As a corollary to Theorem C, we get
\begin{CMTC}
For each 
$t\in(\inf f,\sup f)\setminus{C_V(f)}, $
the level set   $f^{-1}(t)$ consists of locally finitely many mutually disjoint 
curves which are locally bi-Lipschitz homeomorphic  to an interval.
\end{CMTC}
\begin{remark}
S. Ferry \cite{Fe} proved that the critical values of the distance function $d_N$  is of measure zero for a closed subset $N$  of $\R^n, n\leq3,$ and  he constructed  a closed subset  $N$ of $\R^4$ such that  the level set $d_N^{-1}(t)$ is not a submanifold  for all $t\in(0,1).$ 
Fu \cite{Fu} improved his result. 
\end{remark}

\bigskip

Let us recall that  a {\it Finsler manifold} $(M,F)$ is an $n$-dimensional differential manifold $M$ endowed with a norm
$F:TM\to [0,\infty)$ such that
\begin{enumerate}
	\item $F$ is positive and differentiable on $\widetilde{TM}:=
	TM\setminus\{0\}$;
	\item $F$ is 1-positive homogeneous, i.e., $F(x,\lambda y)=\lambda F(x,y)$, $\lambda>0$, $(x,y)\in TM$;
	\item the Hessian matrix $g_{ij}(x,y):=\dfrac{1}{2}\dfrac{\partial^2 F^2}{\partial y^i\partial y^j}$ is positive definite on $\widetilde{TM}.$
\end{enumerate}
Here $TM$ denotes the tangent bundle over the manifold $M.$
The Finsler structure is called {\it absolute homogeneous} if $F(x,-y)=F(x,y)$ because this leads to the homogeneity condition $F(x,\lambda y)=|\lambda| F(x,y)$, for any $\lambda\in \mathbb R$.

By means of the Finsler fundamental function $F$ one defines the {\it indicatrix bundle} (or the Finslerian {\it unit sphere bundle}) by $SM:=\bigcup_{x\in M}S_xM$, where $S_xM:=\{y\in M\ | \ F(x,y)=1\}$.

On a Finsler manifold $(M,F)$ one can  define the integral length of curves as follows. If $\gamma:[a,b]\to M$ is a  $C^{1}$-curve in $M$,
then the
{\it integral length} of $\gamma$ is given by
\begin{equation}\label{integral length}
{\cal L_\gamma}:=
\int_a^b
F(\gamma(t),\dot\gamma(t))dt,
\end{equation}
where $\dot\gamma=\dfrac{d\gamma}{dt}$ denotes the tangent vector along the curve $\gamma$. 

For a  $C^1$-{\it variation} 
\begin{equation}
\bar \gamma:(-\varepsilon,\varepsilon)\times[a,b]\to M,\quad (u,t)\mapsto \bar\gamma(u,t)
\end{equation}
of the base curve $\gamma(t)$, with the {\it variational vector field} 
$U(t):=\dfrac{\partial\bar\gamma}{\partial u}(0,t)$,  
by a straightforward computation
one obtains
\begin{equation}\label{arbitrary first variation} 
({\cal L_\gamma})'(0)=g_{\dot\gamma(b)}(\dot \gamma,U)|_a^b-\int_a^bg_{\dot\gamma}(D_{\dot\gamma}{\dot\gamma},U)dt,
\end{equation}
where $D_{\dot\gamma}$ is the covariant derivative along $\gamma$ with respect to the Chern connection and $\gamma$ is arc length parametrized
(see \cite{BCS}, p. 123, or \cite{S}, p. 77 for details of this computation as well as for some basics on Finslerian connections).

A regular $C^\infty$-curve $\gamma$ on a Finsler manifold is called a {\it geodesic} if $({\cal L}_\gamma)'(0)=0$ for all
$C^1$-variations of $\gamma$ that keep its ends fixed.
In terms of Chern connection a constant speed geodesic is characterized by the condition $D_{\dot\gamma}{\dot\gamma}=0$.

If the base curve $\gamma$ is a geodesic for the variation $\overline\gamma$ above,  
one obtains, by \eqref{arbitrary first variation},
 the following {\it first variation formula}:
\begin{equation}\label{first variation}
({\cal L_\gamma})'(0)=g_{\dot\gamma(b)}({\dot\gamma(b)},U(b))-g_{\dot\gamma(a)}({\dot\gamma(a)},U(a))
\end{equation}
which is fundamental for our present study. 

Using the integral length of a curve, one can define the Finslerian distance between two points on $M$. For any two points $p$, $q$ on $M$, let us denote by $\Omega_{p,q}$ the set of all $C^1$-curves $\gamma:[a,b]\to M$ such that $\gamma(a)=p$ and $\gamma(b)=q$. The map
\begin{equation}
d:M\times M\to [0,\infty),\qquad d(p,q):=\inf_{\gamma\in\Omega_{p,q}}{\cal L}_\gamma
\end{equation}
gives the {\it Finslerian distance} on $M$. It can be easily seen that $d$ is in general a quasi-distance, i.e., it has the properties
\begin{enumerate}
	\item $d(p,q)\geq 0$, with equality if and only if $p=q$;
	\item $d(p,q)\leq d(p,r)+d(r,q)$, with equality if and only if  $r$ lies on a
	minimal geodesic segment joining  $p$ to $q$ (triangle inequality).
\end{enumerate}

In the case where $(M,F)$ is absolutely homogeneous, the symmetry condition $d(p,q)=d(q,p)$ holds and therefore $(M,d)$ is a  metric space. We do not assume this symmetry condition in the present paper.

We recall (\cite{BCS}, p. 151) that a sequence of points $\{x_{i}\}\subset M$, on a Finsler manifold $(M,F)$, is called a {\it forward Cauchy sequence} if for any $\varepsilon>0$, there exists $N=N(\varepsilon)>0$ such that for all $N\leq i < j$ we have $d(x_{i},x_{j})<\varepsilon$. Likely, a sequence of points $\{x_{i}\}\subset M$ is called a {\it backward Cauchy sequence} if for any $\varepsilon>0$, there exists $N=N(\varepsilon)>0$ such that for all $N\leq i < j$ we have $d(x_{j},x_{i})<\varepsilon$.
The Finsler space $(M,F)$ is called {\it forward (backward) complete} with respect to the Finsler distance $d$ if and only if every forward (backward) Cauchy sequence converges, respectively.
Moreover, a Finsler manifold $(M,F)$ is called {\it forward (respectively backward) geodesically complete} if and only if any geodesic $\gamma:(a,b)\to M$ can be forwardly (respectively backwardly) extended to a  geodesic $\gamma:(a,\infty)\to M$ (respectively $\gamma:(-\infty,b)\to M$). 
The equivalence between forward completeness and geodesically completeness is given by the Finslerian version of Hopf-Rinow Theorem (see for eg. \cite{BCS}, p. 168).

Let us point out that in the Finsler case, unlikely the Riemannian counterpart, forward completeness is not equivalent to backward one, except the case when $M$ is compact. A Finsler metric that is forward and backward complete is called {\it bi-complete}.

Let us also recall that for a forward complete Finsler space $(M,F)$, the exponential map $\exp_p:T_pM\to M$ at an arbitrary point $p\in M$ is a surjective map (see \cite{BCS}, p. 152 for details). 

A unit speed geodesic on $M$ with initial conditions $\gamma(0)=p\in M$ and $\dot\gamma(0)=T\in S_pM$ can be written as $\gamma(t)=\exp_p(tT).$
Even though the exponential map is quite similar with the correspondent notion in Riemannian geometry, we point out two distinguished properties (see \cite{BCS}, p. 127 for details):
\begin{enumerate}
	\item $\exp_x$ is only $C^1$ at the zero section of $TM$, i.e. for each fixed $x$, the map $\exp_xy$ is $C^1$ with respect to $y\in T_xM$, and $C^\infty$ away from it. Its derivative at the zero section is the identity map (\cite{W});
	\item $\exp_x$ is $C^2$ at the zero section of $TM$ if and only if the Finsler structure is of Berwald type. In this case $\exp$ is actually $C^\infty$ on entire $TM$ (\cite{AZ}).
\end{enumerate}

\section{1-Lipschitz functions and $f$-geodesics }
In this section we will give  an equivalent condition for a  function on a connected 
Finsler manifold to be 1-Lipschitz and some important examples of a 1-Lipschitz function on a manifold.

From now on,
 $(M,F)$ denotes a connected  Finsler manifold with Finslerian distance function $d.$

\begin{definition}
A function $f$  on the manifold $M$ is said to be {\it 1-Lipschitz} if 
\begin{equation}\label{eq: Lip}
f(y)-f(x)\leq d(x,y)
\end{equation}
 holds for any $x,y\in M.$
\end{definition}



 Let $f$ be a 1-Lipschitz function on the manifold $M.$
By exchanging $x$ and $y$ in the equation \eqref{eq: Lip}, 
we get
that
\begin{equation}\label{eq: Lip2}
-d(y,x)\leq f(y)-f(x)
\end{equation}
holds for all $x,y$ of $M.$
From \eqref{eq: Lip} and \eqref{eq: Lip2} it follows that
$f\circ \varphi^{-1} :\varphi(U)\to \R$
is locally Lipschitz for any local chart $(U,\varphi)$ of $M.$

From  Rademacher's theorem (for example, see \cite{M}), $f\circ\varphi^{-1}$ is differentiable almost everywhere on $\varphi(U),$ and 
hence the Finslerian gradient $\nabla f$ of $f$  exists almost everywhere on $M.$ See \cite{S} for the definition and basic properties of the Finslerian gradient.

Choose any differentiable point $p$ of $f$ and any unit tangent vector $v$ at $p.$
Let $\gamma :[0,a]\to M$ denote the minimal geodesic segment defined by
$\gamma(t):=\exp_p(tv).$
From \eqref{eq: Lip}, we have
$f(\gamma(t))-f(\gamma(0))\leq d(\gamma(0),\gamma(t))=t$
for $t\in(0,a).$
Thus,
\begin{equation}\label{eq: grad f}
df_p(v)=\lim_{t\searrow0}\frac{f(\gamma(t))-f(\gamma(0))}{t}\leq 1.
\end{equation}
If  $df_p\ne 0,$ 
then  we obtain, from \eqref{eq: grad f},
\begin{equation}\label{eq: grad2}
g_{\nabla f_p}(\nabla f_p,v)\leq 1
\end{equation}
for all unit tangent vectors $v$ at $p.$
By substituting $v=\nabla f_p/F(\nabla f_p)$ in the equation \eqref{eq: grad2}
we get $F(\nabla f_p)\leq 1.$
If $ df_p=0,$ then it is trivial that $F(\nabla f_p)=0\leq 1.$
Therefore, $\nabla f$ exists  almost everywhere and 
$F(\nabla f_p)\leq 1 $ if $\nabla f_p$ exists. 

We may prove the converse by imitating the proof of \cite [Lemma 3.13] {KT1}, and hence we get

\begin{proposition}\label{prop: equivalence}
Let $f$ be a function defined on a connected  Finsler manifold $M$.
Then,
the Finslerian gradient
$\nabla f$ of
the function
$f$ exists almost everywhere and $F(\nabla f)\leq 1$   almost everywhere 
if and only if $f$ is 1-Lipschitz.
\end{proposition}

\begin{example}\label{exd_N}

Let $N$ be a closed subset of a connected  Finsler manifold $M.$
We define a function $d_N$ on $M$ 
by
$d_N(x):=\inf\{ d(p,x) |\: p\in N\}.$
The function $d_N$ is a 1-Lipschitz function on $M.$ Indeed,
choose any points $x,y$ of $M$ and any positive number $\epsilon.$
There exists a point $p_\epsilon\in N$ such that 
\begin{equation}\label{eqEd_N.1}
d(p_\epsilon,y)-\epsilon<d_N(y).
\end{equation}
Moreover, it is trivial that
\begin{equation}\label{eqEd_N.2}
d_N(x)\le d(p_\epsilon,x).
\end{equation}
Hence, by \eqref{eqEd_N.1}, \eqref{eqEd_N.2} and the triangle inequality,
we get,
$$ d_N(x)-d_N(y)<d(p_\epsilon,x)-d(p_\epsilon,y)+\epsilon\le d(y,x)+\epsilon.$$
Since $\epsilon$ is arbitrary,
we have
$d_N(x)-d_N(y)\le d(y,x).$
In particular,
if the subset $N$ is a single point $p,$
then the distance function $d_p$ from the point $p$ is 1-Lipschitz.
\end{example}

\begin{example}\label{exd^N}
We can introduce 
the function $d^N$  similar to $d_N$ defined by
$d^N(x):=\inf\{d(x,p)|\: p\in N \}.$
This function is not always a 1-Lipschitz function on $M,$
but $(-1)d^N$ is 1-Lipschitz.
\end{example}

\begin{example}\label{EBu}
A unit speed geodesic $\gamma:[0,\infty)\to M$ in a forward complete Finsler manifold $M$
is called a {\it ray} if $d(\gamma(0),\gamma(t))=t$ holds on $[0,\infty).$
Then, the  function $B_\gamma$  on $M$ is defined by
$B_\gamma(x):=\lim_{t\to\infty} \{t-d(x , \gamma(t))    \}$
is 1-Lipschitz. This function is called  the {\it Busemann function} of the ray $\gamma.$

\end{example}

\begin{example}\label{Emax,min}

Let $f_1$ and $f_2$ be 1-Lipschitz functions on  a connected Finsler manifold  $M.$
Then $f:=\max(f_1,f_2)$ is 1-Lipschitz.
Indeed, choose any points $x,y\in M.$ 
Without loss of generality, we may assume that $f(x)=f_1(x).$
Since $f(y)\ge f_1(y),$
we
obtain,
$f(x)-f(y)\le f_1(x)-f_1(y)\le d(y,x).$
It is also easy to prove  that $\min(f_1,f_2)$ is a 1-Lipschitz function.

\end{example}

\begin{example}\label{EgBu}
The Busemann function can be generalized as follows. Let us consider a divergent sequence of points $\{x_n\}_n$ on the forward complete noncompact Finsler manifold $(M,F)$, that is $\lim_{n\to \infty}d(x_1,x_n)=\infty$.
Then the function $f:M\to \R$ given by 
$$
f(x):=\limsup_{n\to \infty}[d(x_1,x_n)-d(x,x_n)] 
$$
is an 1-Lipschitz function. Notice that  the function $[d(x_1,x_n)-d(x,x_n)] $ is  
1-Lipschitz for each $n,$ since $(-1)d^{x_n}$ is 1-Lipschiz.
This function is called a   {\it horofunction} in \cite{Cu} for Riemannian case.
\end{example}

\begin{example}\label{EWu}
	The following function gives another interesting example of  a 1-Lipschitz  function on a Riemannian manifold, which was introduced in the paper \cite{Wu}.
 Let $M$ be a  complete noncompact Riemannian  manifold and let us define
	\begin{equation}
	\eta:M\to \R,\quad \eta(x):=\limsup_{t\to \infty}\{t-d(x,S_p(t))\},
	\end{equation}
	where $S_p(t):=\{x\in M | \; d(p,x)=t\},$  (see \cite{Wu} for a more general setting).
	The  triangle inequality implies that $\eta_t(x):=t-d(x,S_p(t))\leq d(p,x)$ for all $t>0,$  hence $\eta$ is well-defined. 
	For any fixed $t>0$ the function $\eta_t(x)$ is 1-Lipschitz, hence $\eta$ is also 1-Lipschitz. 
	 
\end{example}

From now on, $f$ always denote a 1-Lipschitz function on the connected Finsler  manifold $M.$

\begin{definition}\label{def f-geodesic}
A unit speed nonconstant geodesic  $\gamma : I\to M$ called an {\it $f$-geodesic} if
\begin{equation}\label{eq f-geodesic}
f(\gamma(t))-f(\gamma(s))=t-s
\end{equation}
for any $s,t\in I,$ where $I $ denotes an interval.
\end{definition}

 Such an $f$-geodesic is called {\it maximal} if there are no $f$-geodesics containing $\gamma(I)$ as a proper subarc. Notice that any $f$-geodesic is minimal. This property is proved in 
 Lemma \ref{lem: Lipschitz curve is f-geod}. Moreover, a sufficient condition for a unit speed   Lipschitz  curve to be an $f$-geodesic is given in this lemma.

\begin{lemma}\label{lem: Lipschitz curve is f-geod}
	If a unit speed Lipschitz curve $\gamma:[a,b]\to M,$ ( i.e. $|\dot \gamma(t)|=1$ a.e.) satisfies
	\begin{equation*}
\begin{split}
	b-a=f(\gamma(b))-f(\gamma(a))
	\end{split}
\end{equation*}
	then $\gamma$ is a minimal  geodesic segment and  an $f$-geodesic. 
In particular, any $f$-geodesic is minimal.
\end{lemma}

\begin{proof}
The length $L(\gamma)$ of    the curve $\gamma$ is given by
$L(\gamma)=\int_a^bF(\dot\gamma(t))dt.$
Since $\gamma$ is unit speed, we get
 $L(\gamma)=b-a.$ Since $f$ is 1-Lipschitz, 
we get
$$L(\gamma)=b-a=f(\gamma(b))-f(\gamma(a))\leq d(\gamma(a),\gamma(b)).$$
Hence the curve  $\gamma$  is a minimal geodesic segment. 

In order to prove that $\gamma$ is an $f$-geodesic,
we will next prove that the function 
$\varphi(t):=t-f(\gamma(t))$ 
is an  increasing function on $ [a,b].$
Choose any $a\leq s<t\leq b.$
Since $f$ is 1-Lipschitz and the geodesic $\gamma$ is a unit speed minimal geodesic, we get
$\varphi(s)-\varphi(t)\leq s-t+d(\gamma(s),\gamma(t))=0.$
Thus, the function $\varphi$ is increasing on $[a,b].$
By combining  our assumption
$\varphi(a)=\varphi(b),$ we may conclude that 
 $\varphi$ is  constant and  the geodesic $\gamma$ satisfies \eqref{eq f-geodesic}  for any $s,t\in [a,b]$.
$\qedd$
\end{proof}

\begin{example}
Let $\gamma:[0,a]\to M$ be  a unit speed minimal geodesic segment emanating from  a point $p=\gamma(0)$ on a connected Finsler manifold $M.$
Then, it is clear that $d_p(\gamma(t))=t $ holds on $[0,a],$ where $d_p(x):=d(p,x)$ for each $x\in M.$ Hence $\gamma$ is a $d_p$-geodesic.
There does not always exist a minimal geodesic joining  the point $p$ to any point of $M,$ since $M$ is not assumed to be complete.
The following lemma  guarantees the local existence of a $d_p$-geodesic without assumption of the completeness of a  Finsler manifold.
\end{example}

\begin{lemma}\label{lem2.13}
Let $N$ be  a closed subset of a connected Finsler manifold $M.$ 
Then for each point $q\notin N,$  there exists a $d_N$-geodesic to $q.$
In particular, for each point $q,$  there exists a $d_p$-geodesic to $p,$
if the point $p$ is distinct from $q.$
\end{lemma}

\begin{proof}
Choose any point $q\notin N$ and a small positive number $\delta,$  so that   
${ \cal S}:=\{ x\in M| \; d(x,q)=\delta \}$ is compact, and each point of $\cal S$ is connected to the point $q$ by the unique minimal geodesic segment.
Since  $\cal S$ is compact and the function $d_N$ is continuous, there exist a point $x_0\in\cal S$ with $d_N(x_0)=\min\{ d_N( x)|\: x\in {\cal S}\}.$
Then it is easy to check that
\begin{equation}\label{eqlem2.12} 
d_N(q)\ge d_N(x_0)+d(x_0,q).
\end{equation} 
If $\alpha: [0,\delta]\to M$ denotes the unit speed minimal geodesic segment emanating from $x_0$ to $q,$ the equation \eqref{eqlem2.12} implies that
$$d_N(\alpha(\delta))\ge d_N(\alpha(0))+\delta,$$
and hence
$$d_N(\alpha(\delta))= d_N(\alpha(0))+\delta,$$
since $d_N$ is 1-Lipschitz.
From Lemma \ref{lem: Lipschitz curve is f-geod} it follows that $\alpha:[0,\delta]\to M$ is a $d_N$-geodesic to  the point $q.$

$\qedd$\end{proof}

%
\section {First variation formulas for  almost distance functions }

We will give the proof of {\bf Theorem A}   and some interesting examples
of an almost distance function in this section.
Let $f$ be a 1-Lipschitz function on a connected  Finsler manifold $(M,F).$

It is sometimes convenient to make use of  the following parametrization of $f$-geodesics. 

\begin{definition}
The parameter 
of an $f$-geodesic $\alpha:[a,b]\to M$ 
is called {\it canonical}
if $f(\alpha(t))=t$  holds on $[a,b].$ 
\end{definition}

By definition, it is clear that an $f$-geodesic is preserved by any parallel translation of the parameter of the $f$-geodesic. Hence, we can introduce 
the canonical parameter for any $f$-geodesic.

\begin{lemma}\label{primitive first variation formula, ver.1}
Let $\{ \gamma_i \}_i $ be a sequence of $f$-geodesics defined on a common interval $[0,a]$. If the following three limits

\begin {equation}\label{eq3.1}
p:=\lim_{i\to \infty}  \gamma_i(0)
\end{equation}
\begin{equation}\label{eq3.2}
w_\infty:=\lim_{i\to\infty}\dot\gamma_i(0)
\end{equation}
\begin{equation}\label{eq3.3}
v^f:=\lim_{i\to\infty}\frac{1}{F(\exp^{-1}_p(\gamma_i(0)) )}
\exp^{-1}_p(\gamma_i(0) )
\end{equation}
exist, then 

\begin{equation}\label{eq3.4}
g_{w_{\infty} }(w_{\infty} ,v^f)\ge  g_{\dot\gamma(f(p) )}(\dot\gamma(f(p)),v^f)
\end{equation}
holds for any $f$-geodesic $\gamma$ emanating from $p$ with canonical parameter.

Moreover, we have
\begin{equation}\label{eq3.5}
\lim_{i\to \infty}\frac{f(\gamma_i(0))-f(p)}{d(p,\gamma_i(0))}=g_{w_\infty}(w_\infty,v^f).
\end{equation}
Here $\exp^{-1}_p$ denotes the local inverse map of the exponential map $\exp_p$ around the zero vector.
\end{lemma}
\begin{proof}
We may assume that each $\gamma_i(0)$ is a point in a convex ball centered at $p.$
For each $\gamma_i(0),$
let $\sigma_i :[0,d(p,\gamma_i(0) )]\to M$ denote the unit speed minimal geodesic
segment joining $p$ to $\gamma_i(0),$ and hence
 
\begin{equation}\label{eq3.6}
\dot\sigma_i(0)=\frac{1}{F(\exp^{-1}_p(\gamma_i(0)) )}\exp^{-1}_p(\gamma_i(0)).
\end{equation}
Let us choose a constant $\delta\in(0,a)$ in such a way that $\gamma_\infty(\delta)$ is a point of   a strongly convex ball around $p.$ 
Here $\gamma_\infty(t):=\exp(tw_\infty)$ denotes the limit geodesic of the sequence $\{\gamma_i\}.$
Since $\gamma_i$ is an $f$-geodesic and $f$ is 1-Lipschitz,
we obtain,
$f(\gamma_i(0))-f(\gamma_i(\delta)  )=-d(\gamma_i(0),\gamma_i(\delta)),$
and
$f(\gamma_i(\delta))-f(p)\le d(p,\gamma_i(\delta)), $ respectively.
Hence,
\begin{equation}\label{eq3.7}
f(\gamma_i(0))-f(p)\le d(p,\gamma_i(\delta))-d(\gamma_i(0),\gamma_i(\delta)).
\end{equation}
It follows from  the first variation formula   \eqref{first variation} that
\begin{equation}\label{eq3.8}
d(p,\gamma_i(\delta))-d(\gamma_i(0),\gamma_i(\delta))=-\int_0^{d_i} h'(t)dt=\int_0^{d_i} g_{w_i(t)}(w_i(t),\dot\sigma_i(t))dt
\end{equation}
holds, where 
$d_i:=d(p,\gamma_i(0)),$ $h(t):=d(\sigma_i(t),\gamma_i(\delta)),$
and $w_i(t)$ denotes the unit velocity vector at $\sigma_i(t)$ of the minimal geodesic segment joining $\sigma_i(t)$ to $\gamma_i(\delta).$
It is clear that for each small interval $[0,c],$ the two sequences $\{w_i(t)\}$ and $\{\dot\sigma_i(t)\}$ of vector valued functions uniformly converge to $w_\infty(t)$ and $\dot\sigma_\infty(t)$ on $[0,c]$ respectively, where
$\sigma_\infty(t):=\exp(tv^f),$ and $w_\infty(t)$ denotes
the unit velocity vector at $\sigma_\infty(t)$ of the minimal geodesic segment joining  $ \sigma_\infty(t)$ to $\gamma_\infty(\delta).$
Therefore, by \eqref{eq3.8},
we obtain,
\begin{equation}\label{eq3.9}
\lim_{i\to \infty}\frac{d(p,\gamma_i(\delta))-d(\gamma_i(0),\gamma_i(\delta))}{d_i}=g_{w_\infty}(w_\infty,v^f).
\end{equation}
Combining \eqref{eq3.7} and \eqref{eq3.9}, we get
\begin{equation}\label{eq3.10}
\limsup_{i\to \infty}\frac{f(\gamma_i(0))-f(p)}{d_i}\le g_{w_\infty}(w_\infty,v^f).
\end{equation}

Let $\beta$ denote any $f$-geodesic emanating from $p$ with canonical parameter.
Since $f$ is 1-Lipschitz, and $\beta$ is an $f$-geodesic,
we have
$f(\gamma_i(0))-f(\beta(f(p)+\delta))\ge -d(\gamma_i(0),\beta(f(p)+\delta)),$
and $f(\beta(f(p)+\delta))-f(p)=d(p,\beta(f(p)+\delta)),$ respectively, 
where $\delta>0$ is a sufficiently small positive number, so that $\beta(f(p)+\delta)$ is a point in a strongly convex ball centered at $p.$
Hence, we obtain,
$f(\gamma_i(0))-f(p)\ge d(p,\beta(f(p)+\delta))-d(\gamma_i(0),\beta(f(p)+\delta)).$
By making use of the same technique as above,
we get
\begin{equation}\label{eq3.11}
\liminf_{i\to \infty}\frac{f(\gamma_i(0) )-f(p)}{d_i}\ge g_{\dot\beta(f(p))}(\dot\beta(f((p)),v^f)
\end{equation}
for any $f$-geodesic $\beta$ emanating from $p.$
In particular, by choosing the limit $f$-geodesic $\beta=\gamma_\infty,$ we get
\eqref{eq3.5}, and \eqref{eq3.4} follows from \eqref{eq3.10} and \eqref{eq3.11}.
$\qedd$
\end{proof}


\begin{lemma}\label{primitive first variation formula, ver.2}
Let $\{ \gamma_i \}_i $ be a sequence of $f$-geodesics defined on a common interval $[b,0]$. If the following three limits

\begin {equation}\label{eq3.12}
p:=\lim_{i\to \infty}  \gamma_i(0)
\end{equation}
\begin{equation}\label{eq3.13}
w_\infty:=\lim_{i\to\infty}\dot\gamma_i(0)
\end{equation}
\begin{equation}\label{eq3.14}
v^f:=\lim_{i\to\infty}\frac{1}{F(\exp^{-1}_p(\gamma_i(0)) )}
\exp^{-1}_p(\gamma_i(0) )
\end{equation}
exist, then 

\begin{equation}\label{eq3.15}
g_{w_{\infty} }(w_{\infty} ,v^f)\le  g_{\dot\gamma(f(p) )}(\dot\gamma(f(p)),v^f)
\end{equation}
holds for any $f$-geodesic $\gamma$ to $p$ with canonical parameter.
Moreover, we have
\begin{equation}\label{eq3.16}
\lim_{i\to \infty}\frac{f(\gamma_i(0))-f(p)}{d(p,\gamma_i(0))}=g_{w_\infty}(w_\infty,v^f).
\end{equation}
\end{lemma}

\begin{proof}
We may assume that each $\gamma_i(0)$ is a point in a convex ball centered at $p.$
For each $\gamma_i(0),$
let $\sigma_i :[0,d(p,\gamma_i(0) )]\to M$ denote the unit speed minimal geodesic
segment joining $p$ to $\gamma_i(0),$ and hence
\begin{equation}\label{eq3.17}
\dot\sigma_i(0)=\frac{1}{F(\exp^{-1}_p(\gamma_i(0)) )}\exp^{-1}_p(\gamma_i(0)).
\end{equation}
Let us choose a constant $\delta\in(0,|b|)$ in such a way that $\gamma_\infty(-\delta)$ is a point of   a strongly convex ball around $p.$ 
Here $\gamma_\infty(t):=\exp(tw_\infty)$ denotes the limit geodesic of the sequence $\{\gamma_i\}.$
Since $\gamma_i$ is an $f$-geodesic and $f$ is 1-Lipschitz,
we obtain,
$f(\gamma_i(0))-f(\gamma_i(-\delta)  )=d(\gamma_i(-\delta),\gamma_i(0)),$
and
$f(\gamma_i(-\delta))-f(p)\ge -d(\gamma_i(-\delta),p).$
Hence,

\begin{equation}\label{eq3.18}
f(\gamma_i(0))-f(p)\ge d(\gamma_i(-\delta), \gamma_i(0))-d(\gamma_i(-\delta),p).
\end{equation}
It follows from  the first variation formula   \eqref{first variation} that
\begin{equation}\label{eq3.19}
d(\gamma_i(-\delta),\gamma_i(0))-d(\gamma_i(-\delta),p)=\int_0^{d_i} h'(t)dt=\int_0^{d_i} g_{w_i(t)}(w_i(t),\dot\sigma_i(t))dt
\end{equation}
holds, where 
$d_i:=d(p,\gamma_i(0)),$ $h(t):=d(\gamma_i(-\delta), \sigma_i(t)),$
and $w_i(t)$ denotes the unit velocity vector at $\sigma_i(t)$ of the minimal geodesic segment joining $\gamma_i(-\delta)$ to $\sigma_i(t).$
It is clear that for each small interval $[0,c],$ the two sequences $\{w_i(t)\}$ and $\{\dot\sigma_i(t)\}$  of vector valued functions uniformly converge to $w_\infty(t)$ and $\dot\sigma_\infty(t)$ on $[0,c]$ respectively, where
$\sigma_\infty(t):=\exp(tv^f)$ and $w_\infty(t)$ denotes
the unit velocity vector at $\sigma_\infty(t)$ of the minimal geodesic segment joining  $ \sigma_\infty(t)$ to $\gamma_\infty(-\delta).$
Therefore, by \eqref{eq3.19},
we obtain,
\begin{equation}\label{eq3.20}
\lim_{i\to \infty}\frac{d(\gamma_i(-\delta),\gamma_i(0))-d(\gamma_i(-\delta),p))}{d_i}=g_{w_\infty}(w_\infty,v^f).
\end{equation}
Combining \eqref{eq3.18} and \eqref{eq3.20}, we get
\begin{equation}\label{eq3.21}
\liminf_{i\to \infty}\frac{f(\gamma_i(0))-f(p)}{d_i}\ge g_{w_\infty}(w_\infty,v^f).
\end{equation}

Let $\beta$ denote any $f$-geodesic to $p$ with canonical parameter.
Since $f$ is 1-Lipschitz, and $\beta$ is an $f$-geodesic,
we have
$f(\gamma_i(0))-f(\beta(f(p)-\delta))\le d(\beta(f(p)-\delta), \gamma_i(0)),$
and 
$f(\beta(f(p)-\delta))-f(p)=-d(\beta(f(p)-\delta),p),$
where $\delta>0$ is a sufficiently small number, so that $\beta(f(p)-\delta)$ is a point in a strongly convex ball at $p.$
Hence, we obtain,
$f(\gamma_i(0))-f(p)\le d(\beta(f(p)-\delta), \gamma_i(0) )-d(\beta(f(p)-\delta), p).$
By making use of the same technique as above,
we get
\begin{equation}\label{eq3.22}
\limsup_{i\to \infty}\frac{f(\gamma_i(0) )-f(p)}{d_i}\le g_{\dot\beta(f(p))}(\dot\beta(f((p))  ),v^f)
\end{equation}
for any $f$-geodesic $\beta$ to $p.$
In particular, by choosing the limit $f$-geodesic $\beta=\gamma_\infty,$ we get
\eqref{eq3.16}, and \eqref{eq3.15} follows from \eqref{eq3.21} and \eqref{eq3.22}.
$\qedd$
\end{proof}

\begin{definition}
A 1-Lipschitz function $f$  on a connected Finsler manifold $M$ called an {\it almost distance function} if for each point $p\in f^{-1}(\inf f,\sup f),$ there exist a neighborhood $U_p$ of $p$ and  a positive constant
$\delta(p)$ such that for each point $q\in U_p,$ there exists an $f$-geodesic $\gamma_q:[0,\delta(p)]\to M$  with $q=\gamma_q(0)$ or $q=\gamma_q(\delta(p)).$ 
\end{definition}

\begin{remark}
Observe that this definition guarantees that maximal $f$-geodesics do not shrink to a point. 
It is possible to define the almost distance function on  a metric space such as
an Alexandrov space or a Busemann E-space. Note that a Finsler manifold is a Busemann E-space (see \cite{Bu1}).
For example, {\bf Theorem B} in the introduction would be true for an almost distance function on an Alexandrov surface. We, however, restrict ourselves to the almost distance function on a differentiable manifold in this paper.

Obviously, the distance function from a closed subset $N$ and the Busemann function on a Finsler manifold are typical examples of almost distance functions (see Examples \ref{Ed_p almost dist}, \ref{EBu almost dist} and \ref{EWu almost dist}
below).
\end{remark}

\begin{example}\label{Ed_p almost dist}
Let $N$ denote a closed subset of a connected Finsler manifold $M.$
From the proof of Lemma \ref{lem2.13}  it is easy to check that
for each $p\notin d_N^{-1}(\inf d_N)=d_N^{-1}(0),$  
any point $q\in B_{r(p)}(p)$ admits a $d_N$-geodesic 
$\gamma_q :[0,r(p)]\to M$ with 
$\gamma_q(r(p))=q.$ 
Here $r(p):=d_N(p)/3,$ and $B_{r(p)}(p)$ denotes the open ball centered at $p$ with radius $r(p).$
Thus, $d_N$ is an almost distance function on $M.$
\end{example}

\begin{example}\label{EBu almost dist}
Let $\gamma:[0,\infty)\to M$ denote a ray  on a forward complete Finsler manifold $M.$ Then, a ray $\sigma:[0,\infty)\to M$ on $M$ is called a {\it coray}  of (or {\it asymptotic} ray to) the ray $\gamma$  if $\sigma$ has a sequence $\{q_j\}$ of points of $M$ and a sequence $\{t_i\}$  of positive numbers with $\lim_{i\to \infty} { t_i}=\infty$
such that $\sigma(0)=\lim_{i\to \infty}q_i$ and $ \dot\sigma(0)=\lim_{i\to\infty}\dot\sigma_i(0),$ where
$\sigma_i$ denotes a minimal geodesic segment joining $q_i=\sigma_i(0)$ to $\gamma(t_i).$ 
Note that Busemann \cite{Bu2} introduced the notion of asymptotic rays on a Busemann E-space, and he proved in 
\cite[Theorem 11.2] {Bu1}
that  a ray $\sigma :[0,\infty)\to M$  is a coray of a ray  $\gamma$ if and only if the ray $\sigma$ is   a $B_\gamma$-geodesic on a Busemann E-space  $M.$ 
Note that a Finsler manifold is a Busemann E-space.
It is checked again  in \cite{Sa}  that each coray is an $B_\gamma$-geodesic, where 
$B_\gamma$ denotes the Busemann function of $\gamma$ defined in Example \ref{EBu}.
Remark that  this property is also checked  in \cite{Oh} under a little stronger definition of the coray on a Finsler manifold.
Therefore, any Busemann function on  a forward complete Finsler manifold  is an almost distance function.
\end{example}

\begin{example}\label{EWu almost dist}
Let us check that  $\eta$ is an almost distance function, where 
 $\eta$  is the function  introduced in Example \ref{EWu}.
It is proved in \cite[Lemma 6]{Wu} that for each point $x\in M$ there exists a ray
$\gamma$  emanating from $x$ satisfying $\eta(\gamma(t))=t+\eta(x).$ This implies that each point admits an $\eta$-geodesic
which is a ray, hence $\eta$ is an almost distance function.
The horofunction $f$  introduced in Example \ref{EgBu} is also an almost distance function.  In fact, for each point $p$ on a complete Riemannian  manifold,  there exists a  ray emanating from $p$ which is an $f$-geodesic. Remark that the $f$-geodesic is  a limit ray of the sequence of the minimal geodesic segments from $p$ to the point $x_n.$

\end{example}

The definition of an almost distance function would lead to the following  question:\\ 
{\it Under what conditions an almost distance function becomes the  distance function from a closed set?}\\
In Proposition \ref{sufficient condition for an almost distance function to a distance function}, we will see a sufficient condition for an almost distance function to be  the distance function from a closed subset.
\begin{example}
An almost distance function on a forward complete and connected Finsler manifold $M$ 
is called a {\it generalized Busemann} function if each point of $M$ admits an $f$-geodesic $\gamma :[0,\infty)\to M$ emanating from $p.$
Remark that  $\gamma$ is a ray by Lemma \ref{lem: Lipschitz curve is f-geod}.
Choose  any two generalized Busemann functions $f_1$ and $f_2$ on a forward complete and connected Finsler manifold $M,$ and put $f:=\max(f_1,f_2).$
From Example \ref{Emax,min}, the function $f$  is 1-Lipschitz.
We will prove that $f$ is a generalized Busemann function.
Let $p$ be any point of $M.$
We may assume that $f_1(p)\geq f_2(p).$
Let $\gamma_1 :[0,\infty)\to M$ denote the maximal  $f_1$-geodesic issuing from $p.$
Hence $f_1(\gamma_1(t))=t+f_1(p)$  holds on $[0,\infty).$
Since $f_2$ is 1-Lipschitz, $f_2(\gamma_1(t))-f_2(p)\leq t$ on $[0,\infty).$
Therefore,
$f_2(\gamma_1(t))\leq t+f_1(p)=f_1(\gamma_1(t))$ on  $[0,\infty).$
This implies
that $f(\gamma_1(t))=f_1(\gamma_1(t))=t+f(p)$ on $[0,\infty),$
 and $\gamma_1$ is an $f$-geodesic issuing from $p.$
In particular, the function $f$ is an almost distance function on $M.$

\end{example}

\begin{proposition}\label{generalized first variation formula for almost distance functions}
Let $f$ be an almost distance function on a connected Finsler manifold $(M,F).$
Let $\alpha$ be  a unit speed geodesic emanating from a point
$p=\alpha(0)$  in $f^{-1}(\inf f,\sup f).$ 
If  the point $p$ admits a    unique maximal $f$-geodesic $\gamma$ with
canonical parameter, then

\begin{equation}\label{eqN.23}
\lim_{t\searrow 0}\frac{f\circ\alpha(t)-f\circ\alpha(0)}{t}=g_{\dot\gamma(f(p))}(\dot\gamma(f(p)),\dot\alpha(0))
\end{equation} 
holds.
\end{proposition}

\begin{proof}
Since $f$ is an almost distance function, there exist a neighborhood  $U_p$ of $p$ and a constant $\delta(p)>0$ such that 
for each point $q\in U_p,$ there exists an $f$-geodesic $\gamma_q:[0,\delta(p)]\to M$ with $q=\gamma_q(0)$ or $\gamma_q(\delta(p)).$
Hence, for each sufficiently small $t>0,$ there exists an $f$-geodesic $\gamma_t:[0,\delta(p)]\to M$ with $\alpha(t)=\gamma_t(0)$ or $\gamma_t(\delta(p)).$
Since the geodesic segment $\gamma$ is a unique $f$-geodesic  through $p$, both tangent vectors $w_\infty$ defined in Lemmas \ref{primitive first variation formula, ver.1} and \ref{primitive first variation formula, ver.2} equal  $\dot\gamma(f(p)).$
Therefore, we get \eqref{eqN.23}.

$\qedd$
\end{proof}
\begin{remark}
Innami \cite{In1} gets a similar result to Proposition \ref{generalized first variation formula for almost distance functions} for Busemann functions on a Busemann G-space.
\end{remark}

{\bf  Proof of Theorem A}\\
Suppose that the point $p$ admits a unique maximal $f$-geodesic. 
From  Proposition \ref{generalized first variation formula for almost distance functions} it follows that
\begin{equation}\label{eq4.4}
\lim_{s\searrow 0}\frac{f(\alpha(s))-f(\alpha(0))}{s}=
g_{\dot\gamma( f(p) )}(\dot\gamma(f(p) ) ,\dot\alpha(0) )
\end{equation}
holds for any unit speed geodesic $\alpha(s)$ emanating from $p=\alpha(0).$
Then,
it follows from Lemma 2.4 in \cite{ST} that $f$ is differentiable at $p,$ and it is trivial  from \eqref{eq4.4} that 
\eqref{diff f} holds.

Suppose next that $f$ is differentiable at $p.$
Let $\alpha$ be any $f$-geodesic  through $p.$
Since $f(\alpha(t))=t$ holds on $[f(p)-\delta(p),f(p)]$ or $[f(p),f(p)+\delta(p)],$
we get
\begin{equation}\label{eqKei Kondo}
g_{\nabla f_p}(\nabla f_p,\dot\alpha(f(p)))=df_p(\dot\alpha(f(p) ) )=1.
\end{equation}
On the other hand, the Cauchy-Schwartz inequality for Finsler norms (see for instance \cite[Lemma 1.2.3]{S}) says that
$$g_{Y} (Y,V)\leq F(Y)F(V),$$
holds for any nonzero vectors $Y , V$ and 
equality holds if and only if $Y=\lambda V$ for some $\lambda>0.$
By \eqref{eqKei Kondo}, equality holds in our case.
Hence, ${\dot\alpha} (f(p))={\nabla f}_p$, 
that is,
 the point $p$ admits   a unique $f$-geodesic  with the velocity  vector $\dot\alpha(f(p))={\nabla }f_p$. 
$\qedd$\\
\bigskip

As corollaries to Theorem A, we get that
\begin{enumerate}
\item[(i)] a given point $p$ of a connected and forward complete Finsler manifold $M$ is a differentiable point of the  Busemann function of a ray $\gamma$ if and only if $p$ admits a unique coray of $\gamma$, and that
\item[(ii)]  a  point $q$   of a connected and backward complete Finsler manifold $M$ is differentiable for the  squared distance function $d_N{}^2$
from a  closed subset $N$ of $ M$ if and only if $q$ admits a unique $d_N$-geodesic from $N$ or $q\in N.$
\end{enumerate}

Indeed, the first claim  was proved for the Riemannian case in \cite{KT1}, and in 
\cite{Sa}  for the Finslerian case.  Likewise, the second claim was
proved  for a closed subset of a complete Riemannian manifold in \cite{MM}, by 
Mantegazza and Mennucci, and in \cite{ST} for the Finslerian case.



\section{ Characterizations of $f$-geodesics}
In this section, some characterizations of $f$-geodesics are given.
It was proved  by Sabau  \cite{Sa} that  a ray $\sigma : [0,\infty)\to M$ is a $B_\gamma$-geodesic, i.e. a coray of $\gamma,$ if and only if  the subarc of $\sigma$ in $B_\gamma^{-1}(-\infty,a]$ is a {\it reverse  $B_\gamma^{-1}[a,\infty)$-segment} for any  $a>B_\gamma(\sigma(0)).$
Here $B_\gamma$ denotes the Busemann function of $\gamma,$ and 
the definition  of the {\it reverse $B_\gamma^{-1}[a,\infty)$-segment} is given  for general closed subsets $N$ in the following.

\begin{definition}
A unit speed geodesic $\gamma:[a,b]\to M$ is called a {\it reverse $N$-segment}
for a closed subset $N$ of a connected Finsler manifold $M,$  if $\gamma$ is a $(-1)d^N$-geodesic and $\gamma(b)\in N,$ i.e., 
$d^N(\gamma(b-t))=t$  holds on $[a,b],$
where $d^N$  is the function defined in Example  \ref{exd^N}.
\end{definition}

Remark that 
 a $d_N$-geodesic $\gamma: [a,b]\to $  is called an  $N$-{\it segment}
in \cite{ST}, if $\gamma(a)\in N,$ i.e., $d_N(\gamma(t))=t-a$ on $[a,b],$ where
the function $d_N$ is defined in Example \ref{exd_N}, and that
our  reverse $N$-segments are called {\it forward }$N$-{\it segments} in \cite{Sa}.
Notice also that 
if we  reverse the parameter of a reverse $N$-segment and  the Finsler structure $F,$  then we get  an $N$-segment with respect to the reversed Finsler structure.
Here the reversed Finsler structure $\overline F$ of F means $\overline F(v):=F(-v).$
\bigskip

Let   $f:M\to \R$ be  a 1-Lipschitz function on a connected Finsler manifold $(M,F).$

\begin{lemma}\label{lem: f-geod are smooth at int points} 
Let $\alpha : [a,t_{0}] \to M$ and $\beta : [t_{0},b] \to M$ denote two $f$-geodesics satisfying $\alpha(t_{0})=\beta(t_{0})$. Then the broken geodesic 
$\gamma: [a,b]\to M$
obtained from  $\alpha$ and $\beta$ is an $f$-geodesic, i.e. $\gamma$ is smooth at $\gamma(t_{0})=\alpha(t_{0})=\beta(t_{0})$.
Hence, any $f$-geodesic emanating from an interior point of a maximal $f$-geodesic $\gamma$ is a subarc of $\gamma.$
\end{lemma}

\begin{proof}
Since $\alpha$ and $\beta$ are $f$-geodesics, we have
\[
f(\alpha(t_0))-f(\alpha(a))=f(\beta(t_0))-f(\alpha(a))=t_0-a,\quad
f(\beta(b))-f(\beta(t_0))=b-t_0,
\]
and by summing up these two relations we obtain
\begin{equation}\label{rel 1}
f(\gamma(b))-f(\gamma(a))=  
f(\beta(b))-f(\alpha(a))
=b-a
\end{equation}
and the conclusion follows from Lemma \ref{lem: Lipschitz curve is f-geod}.
$\qedd$
\end{proof}
Combining 
Theorem A and Lemma \ref{lem: f-geod are smooth at int points},
we get that
any $f$-geodesic $\gamma$ is the unique solution of the differential equation
$\nabla f_{\gamma(t)}=\dot\gamma(t)$  on the interior points of $\gamma.$

\begin{lemma}\label{lem:MT1}
Let $\alpha:[a,b]\to M$ be an $f$-geodesic with canonical parameter on the Finsler manifold $(M,F)$.
For each $t_0\in (a,b)$ the following statements are true
\begin{enumerate}
\item $\alpha |_{[a,t_0]}$ is an  $M_f^{a}$-segment to $\alpha(t_0),$ where
  $M_f^{a}:=f^{-1}(-\infty, a]$;
\item $\alpha |_{[t_0,b]}$ is a reverse
$ ^{b}M_f$-segment emanating from $\alpha(t_0),$ where
$ ^{b}M_f:=f^{-1}[b,\infty)$.
\end{enumerate}
\end{lemma}

\begin{proof}
1. Let us choose arbitrary $t_0\in (a,b).$
Let $x$ be any point of the set $M_f^{a}$, that is, $x\in M$ satisfies $f(x)\leq f(\alpha(a))=a$.
Taking into account that the function $f$ is 1-Lipschitz and $\alpha$ is  an $f$-geodesic, it follows
\begin{equation}
d(x,\alpha(t_0))\geq f(\alpha(t_0))-f(x)\geq f(\alpha(t_0))-f(\alpha(a))=d(\alpha(a), \alpha(t_0)).
\end{equation}
Thus, we have proved that $d(x,\alpha(t_0))\geq d(\alpha(a), \alpha(t_0))$  for any $x\in M_f^{a}$, and  $d_{M^a_f}(\alpha(t_0))=d(\alpha(a), \alpha(t_0))=t_0-a.$
Hence $\alpha|_{[a,t_0]}$ is an $M_f^{a}$-segment by Lemma \ref{lem: Lipschitz curve is f-geod}.

The proof of 2 is completely similar.

$\qedd$
\end{proof}

\begin{lemma}\label{lem: end points level sets}
Let $\alpha:[a,b]\to M$ be an $f$-geodesic with canonical parameter on the Finsler manifold $(M,F)$.
Then $\alpha$ is an  $M_f^{a}$-segment ending at $\alpha(b)$ and a reverse $ ^{b}M_f$-segment emanating from $\alpha(a)$.
In particular, for any $t_0\in (a,b)$, $\alpha|_{[a,t_0]}$ is a unique   $M_f^{a}$-segment to $\alpha(t_0),$  and $\alpha|_{[t_0,b]}$ is a unique reverse  $ ^{b}M_f$-segment emanating from $\alpha(t_0).$
\end{lemma}

\begin{proof}
Let $\{ \ve_i  \}_i$ be a sequence of positive numbers convergent to zero. By Lemma \ref{lem:MT1}, for each $i$, $\alpha|_{[a,b-\varepsilon_i]}$ is an  $M_f^{a}$-segment ending at $\alpha(b-\ve_i)$, and  $\alpha|_{[a+\varepsilon_i,b]}$ is a reverse $ ^{b}M_f$-segment emanating from $\alpha(a+\ve_i)$, respectively.
By taking the limit of the sequence $\{ \ve_i  \}_i$ we can conclude that $\alpha$ is an $M_f^{a}$-segment ending at $\alpha(b)$, and a reverse $ ^{b}M_f$-segment emanating from $\alpha(a)$.
Let us choose an arbitrary $t_0\in (a,b)$. Since $\alpha(t_0)$ is an interior point of $\alpha$ it follows that $\alpha|_{[a,t_0]}$ is the unique  $M_f^{a}$-segment and $\alpha|_{[t_0,b]}$ is the unique reverse
$ ^{b}M_f$-segment.

$\qedd$
\end{proof}

\begin{lemma}\label{prop: local equiv of segments and f-geodesics}
Let $\alpha : [a_1,b_1]\to M$ be  an $M^{a_1}_f$-segment  to $\alpha(b_1).$ Suppose that for each $t\in(a_1,b_1)$ which is sufficiently close to $b_1,$ there exists an $f$-geodesic $\gamma_t :[c,d]\to M$ to $\alpha(t)=\gamma_t(d)$ which intersects $M^{a_1}_f.$ Then $\alpha$ is an $f$-geodesic.
\end{lemma}

\begin{proof}
Choose any $t\in(a_1,b_1)$ sufficiently close to $b_1$  so that  there exists an $f$-geodesic $\gamma_t$ to $\alpha(t)$ which intersects $M^{a_1}_f.$
From Lemma \ref{lem:MT1}, it follows that the subarc of $\gamma_t$ lying in
 ${}^{a_1}M_f=f^{-1}[a_1,\infty)$ is an $M_f^{a_1}$-segment to $\alpha(t).$ 
Since $\alpha$ is also an $M^{a_1}_f$-segment and $\alpha(t)$ is an 
interior point of the  $M^{a_1}_f$-segment $\alpha,$
$\alpha|_{[a_1,t]}$ 
coincides with the subarc of $\gamma_t,$ and in particular
 $\alpha|_{[a_1,t]}$
 is an  $f$-geodesic. Since the parameter value $t$ can be chosen arbitrarily close to $b_1,$ $\alpha$ is an $f$-geodesic.

$\qedd$
\end{proof}
The following lemma is dual to Lemma \ref{prop: local equiv of segments and f-geodesics}. The proof is similar.

\begin{lemma}\label{dual to Lem3.7}
Let $\alpha : [a_1,b_1]\to M$ be  a reverse ${}^{b_1}M_f$-segment emanating from $\alpha(a_1).$ Suppose that for each $t\in(a_1,b_1)$ which is sufficiently close to $a_1,$ there exists an $f$-geodesic $\gamma_t :[c,d]\to M$ emanating from  $\alpha(t)=\gamma_t(c)$ which intersects ${}^{b_1}M_f.$ Then $\alpha$ is an $f$-geodesic.
\end{lemma}
Combining Lemmas \ref{lem: end points level sets} and \ref{prop: local equiv of segments and f-geodesics}, we obtain,
\begin{proposition}\label{Prop3.14}
Let $\alpha : [a_1,b_1]\to M$ be a unit speed geodesic with $f(\alpha(a_1))=a_1.$
Then, $\alpha$ is an $f$-geodesic if and only if $\alpha$ is an $M^{a_1}_f$-segment
which admits a small positive $\delta_1$ such that for
each point $ t\in(b_1-\delta_1,b_1)\subset(a_1,b_1),$
there exists an $f$-geodesic $\gamma_t$ to $\alpha(t)$ which intersects $M^{a_1}_f.$

\end{proposition}
Combining Lemmas \ref{lem: end points level sets} and \ref{dual to Lem3.7}, we obtain,
\begin{proposition}\label{dualProp3.14}
Let $\alpha : [a_1,b_1]\to M$ be a unit speed geodesic with $f(\alpha(b_1))=b_1.$
Then, $\alpha$ is an $f$-geodesic if and only if $\alpha$ is a reverse ${}^{b_1}M_f$-segment
which admits a small positive $\delta_1$ such that for
each point $ t\in(a_1,a_1+\delta_1)\subset(a_1,b_1),$
there exists an $f$-geodesic $\gamma_t$ to $\alpha(t)$ which intersects $^{b_1}M_f.$
\end{proposition}


\section{The singular locus of almost distance functions}
Before defining  the {\it singular locus} of an almost distance function, let us review the definitions of a cut point  and the cut locus of  a point.

Let $p$ be a  point of a  Finsler manifold $M.$
 A point $q$ is called
a  {\it cut point} of $p$  along a minimal geodesic segment $\gamma$ joining  the point $p$ to $q,$
if the geodesic $\gamma$ has no extension beyond $q$ as a minimal geodesic segment.  
The {\it cut locus} of the point $p$ is, by definition, the set of cut points along all minimal geodesic segments emanating from $p.$
Notice that the cut locus  is similarly defined for  a closed subset $N$ of a Finsler manifold by making use of an $N$-segment (see \cite {ST}).
In other words,  a cut point of the point $p$ is an end point of a maximal $d_p$-geodesic
and the cut locus of $p$ is  the set of cut points  with respect to   all maximal $d_p$-geodesics.
Thus, it is natural to define the {\it singular locus} ${\cal C}(f)$ of an almost distance function $f$ on a Finsler manifold $M$ as follows.

\begin {definition}
The set 
$\cC(f) := \{ q \in M | \; q $ is an end point of a maximal
$f$-geodesic$\}$ is called the {\it singular locus} of the almost distance function $f.$
\end{definition}

For example, 
${\cal C}(d_p)=C_p\cup\{p\},$
where $C_p$ denotes the cut locus of $p,$
when $d_p$ is the distance function from a point $p$ on a forward complete Finsler manifold.
For each cut point $q$ of $p,$ there is no $d_p$-geodesic emanating from $q, $ 
and there is no $d_p$-geodesic ending at $p.$
From this observation, we notice that
 each point of the singular locus is divided into two types:

\begin{definition}\label{def upper lower singular}
The subset 
$\cC_{+}(f) := \{ p \in \cC(f)|\; $ there is no $f$-geodesic emanating from $p \}$
 is called the {\it upper singular locus} of the almost distance function $f,$ and \\
the subset
$\cC_{-}(f) := \{ p \in \cC(f) |\; $ there is no $f$-geodesic ending at $p \}$
 is called the {\it lower singular locus} of the almost distance function $f.$
\end{definition}

\begin{remark}
\begin{enumerate}
\item
For the distance function $d_p$ from a point $p\in M,$
${\cal C}_+(d_p)$ equals  the cut locus of $p,$ and ${\cal C}_-(d_p)=\{p\}.$
\item It is easy to see that
$\cC_{+}(f) \cap \cC_{-}(f)=\emptyset$ by Lemma \ref{lem: f-geod are smooth at int points}.
\item
For the Busemann function $B_\gamma$ on a bi-complete Finsler manifold,
 there is no $B_\gamma$-geodesic ending at each point of the singular locus $\cC(B_\gamma)$ of $B_\gamma$
which is called a {\it copoint} of the ray $\gamma.$
Hence, 
 $\cC_-(B_\gamma)=\cC(B_\gamma).$ 
Nasu \cite{N1} introduced first the notion of {\it copoint} for a Busemann E-space.
The copoint is called an  {\it asymptotic conjugate point} in \cite{N1,N2}.
\end{enumerate}
\end{remark}


\bigskip
One natural question to ask is when an almost distance function becomes a distance one. The answer to this question is given in the following proposition. 

\begin{proposition}\label{sufficient condition for an almost distance function to a distance function}
 Let $f$ be an almost distance function on a  backward complete  and connected Finsler manifold $M$ such that ${\cC}_-(f)\subset N:=f^{-1}(c)\neq \emptyset$, where $c:=\inf f.$
If  $f^{-1}[c,\sup f)$ is dense in $M,$
then 
	\begin{equation}\label{almost dist is dist}
	f=d_N+c,
	\end{equation}
holds
on $M,$
	where  $d_N(x):=d(N,x)$. 
\end{proposition}
\begin{proof}
Choose any $x\in f^{-1}(c,\sup f).$
Since $f$ is an almost distance function,
there exists a maximal $f$-geodesic $\gamma : I\to M$ through $x\in \gamma(I).$
Let the parameter of $\gamma$ be canonical, so that $f\circ \gamma(t)=t$ holds on the interval $I.$
Since $f\circ\gamma(t)=t\ge c$ for any $t\in I,$
$\inf I\ge c.$
Choose any $s, t\in I$ with $s<t.$
By Lemma \ref{lem: Lipschitz curve is f-geod}, we obtain
$d(\gamma(s),\gamma(t))=t-s.$
This implies that the sequence $\{ \gamma(t_i)\}$ is a backward Cauchy sequence for each decreasing sequence $\{t_i\}$ of elements of $I$ convergent to $a:=\inf I.$
Therefore, $\lim_{t\searrow a}\gamma(t)=\gamma(a)\in{\cal C}_-(f)$  exists, since $M$ is backward complete. 
Since $N\supset {\cal C}_-(f),$   we have $\gamma(a)\in N,$ and  $a=f(\gamma(a))=c.$
Since $\gamma|_{[c,f(x)]}$ is an $N$-segment by Lemma \ref{lem: end points level sets},
$d_N(x)=L(\gamma|_{[c,f(x)]})=f(x)-c.$
Hence $f=d_N+c$ holds on $f^{-1}[c,\sup f).$
Since the subset $f^{-1}[c,\sup f)$ is dense in $M,$
and since both functions $d_N$ and $f$ are continuous on $M,$ $f=d_N+c$ holds on $M.$


	$\qedd$
\end{proof}



We specialize our discussion in order to obtain  characterizations of the singular locus $\cC(f)$  in terms of  sublevel  or superlevel sets of an almost distance function
$f.$

\begin{lemma}\label{lem:f-geod from cut locus}
Let $f$ be an almost distance function on a forward  complete and connected Finsler manifold $M.$
If $p $ is  a point in $\cC_{+}(f) \cap f^{-1}(\inf f, \sup f),$ then no sequences of points of ${\cal C}_-(f)$ converge  to $p,$ and hence
there exist a neighborhood $W_p(\subset U_p)$ of $p$ and a number $\delta(p) > 0$ such that for each $q \in W_p$, there exists an $f$-geodesic $\gamma : [0,\delta(p)] \to M$ with $\gamma(\delta(p))=q$.
\end{lemma}

\begin{proof}
Since $f$ is an almost distance function, the point $p$ admits a neighborhood $U_{p}$ and a number $\delta(p) > 0$ such that  for each $q \in U_{p}$, there exists an $f$-geodesic $\gamma : [0, \delta(p)] \to M$ with $q=\gamma(0)$ or $q=\gamma(\delta(p))$.
Since $p \in \cC_{+}(f)$, there exists an $f$-geodesic $\alpha : [0, \delta(p)] \to M$ with $\alpha(\delta(p)) = p$, and there does not exist any $f$-geodesic emanating from $p$.

Supposing that there exists a sequence $\{ q_{i}\}$ of points in $\cC_-(f)$ convergent to $p,$ we will get a contradiction.
 Without loss of generality, we may assume that $q_i\in U_p$ for each $i.$
Since $q_i\in \cC_-(f)\cap U_p$  for each $i,$ there exists an $f$-geodesic 
 $ \gamma_{i} : [0, \delta(p)] \to M$
with $\gamma_{i}(0)=q_{i}.$ 
Thus, the sequence $\{\gamma_i\}$ has 
 a limit $f$-geodesic $\gamma_{\infty}: [0, \delta(p)] \to M$ emanating from $\gamma_{\infty}(0)=p.$  This is a contradiction.
Hence, there exists a neighborhood $W_p (\subset U_{p})$ of $p$ such that for each $q\in W_p$, there exists an $f$-geodesic
$\gamma : [0,\delta(p)] \to M$ with $q=\gamma(\delta(p))$.
$\qedd$
\end{proof}

The following lemma is dual to Lemma \ref{lem:f-geod from cut locus}.
\begin{lemma}\label{dual to Lemma3.16}
Let $f$ be an almost distance function on a backward  complete  and connected Finsler Manifold $M.$
If $p $ is  a point in $\cC_{-}(f) \cap f^{-1}(\inf f, \sup f)$, then no sequences of points of ${\cal C}_+(f)$ converge  to $p$ and hence 
there exist a neighborhood $W_p(\subset U_p)$ of $p$ and a number $\delta(p) > 0$ such that for each $q \in W_p$, there exists an $f$-geodesic $\gamma : [0,\delta(p)] \to M$ with $\gamma(0)=q$.
\end{lemma}


\begin{theorem}\label{thm:cut locus and distance properties}
Let $f$ be an almost distance function on a forward complete and connected  Finsler manifold $M.$
If $p$ is a point in $\cC_{+}(f) \cap f^{-1}(\inf f, \sup f)$, then there exist a neighborhood $V_p(\subset W_p)$ of $p$ and a positive number $\delta(p)$ such that
for any unit speed geodesic segment
$\alpha : [a_0,b_0]\to M$ with $\alpha(b_0)\in V_p$ and with 
$\alpha(a_0)\in f^{-1}(a_0),$
$\alpha$ is an f-geodesic if and only if $\alpha$ is an  $M^{a_0}_f$-segment. Here $a_0:=-1/2\cdot\delta(p)+f(p).$
In particular
\begin{equation}\label{f-cut locus coincides with N_f cut locus}
{\cC}({M}^{a_0}_{f}) \cap V_p = \cC_{+}(f) \cap V_p.
\end{equation}
 Here  ${\cal C}(M^{a_0}_f)$ denotes the cut locus of the closed subset 
$M^{a_0}_f$ of $M.$
\end{theorem}

\begin{proof}
By  Lemma \ref{lem:f-geod from cut locus}, there exist a neighborhood $W_p$ of $p$ and a constant $\delta(p)>0$ such that for each $q\in W_p$ there exists an $f$-geodesic $\gamma_q : [0,\delta(p)]\to M$  to $q=\gamma_q(\delta(p)).$ Thus the geodesic $\gamma_q$ intersects $M^{a_0}_f,$ if $f(q)<a_0+\delta(p).$

Let us put 
$$V_p :=W_p \cap f^{-1}(a_0,a_0+\delta(p)).$$
Let $\alpha : [a_0,b_0]\to M$ denote a unit speed geodesic segment with $\alpha(b_0)\in V_p$ and with $\alpha(a_0)\in f^{-1}(a_0).$
Since $V_p$ is open, it is clear that $\alpha(t)\in V_p$ if $t\in(a_0,b_0)$ is sufficiently close to $b_0.$ Therefore,  from Lemma \ref{lem:f-geod from cut locus} it follows that
for each $t\in(a_0,b_0)$ sufficiently close to $b_0,$
there exists an $f$-geodesic $\gamma_t :[0,\delta(p)]\to M$ with $\alpha(t)=\gamma_t(\delta(p)),$ which intersects $M^{a_0}_f,$ and 
it follows from Proposition \ref{Prop3.14}   that $\alpha$ is an $M^{a_0}_f$-segment if and only if $\alpha$ is an $f$-geodesic. 
Hence, the equation \eqref{f-cut locus coincides with N_f cut locus} is clear.

$\qedd$
\end{proof}

Since the Finsler distance function $d(\cdot,\cdot)$ is not always symmetric, we need the notions of the reverce cut points and the reverse cut locus of a closed subset $N.$
For a closed subset $N$ of a connected Finsler manifold $M,$ a {\it reverse cut point} of $N$ is, by definition,  the end point of  a maximal reverse $N$-segment emanating from this point, and the {\it reverse cut locus} of $N$ is the set of all reverse cut points of $N$ along all reverse $N$-segments.
Thus, the reverse cut locus of $N$ equals the lower singular locus of the almost distance  function $(-1)d^N.$
\begin{theorem}\label{thm:cut locus and distance properties 2}
Let $f$ be an almost distance function on a backward complete and connected  Finsler manifold $M.$
If $p$ is a point in  $ \cC_{-}(f) \cap f^{-1}(\inf f, \sup f)$, then there exist a neighborhood $V_p(\subset W_p)$ of $p$ and a positive number $\delta(p)$ such that 
for any unit speed geodesic segment $\alpha : [a_0,b_0]\to M$ with $\alpha(a_0)\in V_p$ and with $\alpha(b_0)\in f^{-1}(b_0),$ $\alpha$ is an $f$-geodesic if and only if $\alpha$ is a reverse $ ^{b_0}M_f$-segment. Here $b_0:=1/2\cdot\delta(p)+f(p).$
In particular,
\begin{equation}\label{f-cut locus coincides with N_f cut locus 2}
\cC_{rev}( ^{b_0}{M}_{f}) \cap V_p = \cC_{-}(f) \cap V_p.
\end{equation}
  Here $\cC_{rev}( ^{b_0}{M}_{f}) $ denotes the reverse cut locus of  the closed subset $ ^{b_0}{M}_{f}$ of $M.$
\end{theorem}

\begin{lemma}\label{lem3.20} 
If $p\in f^{-1}(\inf f,\sup f)$ is a point of  a bi-complete and connected Finsler manifold $M,$
then there exists a neighborhood $V_p$ of $p$
such that for  any point $q\in V_p,$
there exists an $f$-geodesic $\gamma :[0,\delta(p)]\to M$ with $q=\gamma(0)$ or $\gamma(\delta(p)),$ which intersects $^{b_0}M_f$ or $M^{a_0}_f$ respectively. 
In particular,
for each point $q\in V_p,$
\begin{equation}\label{distance to the sublevel}
d({M}^{a_0}_{f},q) = f(q) - a_0
\end{equation}
or
\begin{equation}\label{distance to the sublevel inverse 2}
d(q,\; ^{b_0}{M}_{f}) = b_0 -f(q).
\end{equation}
Here $a_0:=-1/2\cdot\delta(p)+f(p), b_0:=1/2\cdot\delta(p)+f(p).$
\end{lemma}
\begin{proof}
Since $f$ is an almost distance function, there exist a neighborhood $U_p$ of $p$ and a constant $\delta(p)>0$ such that for each $q\in U_p,$
there exists an $f$-geodesic $\gamma_q: [0,\delta(p)]\to M$ with $q=\gamma_q(0)$ or $\gamma_q(\delta(p)).$ Hence if $q\in V_p:=U_p\cap f^{-1}(a_0,b_0),$ then the geodesic $\gamma_q$ intersects $M^{a_0}_f$ or $^{b_0}M_f.$

 Choose any point $q\in V_p$ and fix it.
Hence there exists an $f$-geodesic $\gamma:[0,\delta(p)]\to M$ with $q=\gamma(0)$ or $\gamma(\delta(p)).$ If the geodesic segment $\gamma$ satisfies $\gamma(\delta(p))=q$ (respectively $\gamma(0)=q$), then by Lemma \ref{lem: end points level sets}, the subarc of $\gamma$ lying $f^{-1}[a_0,b_0]$ is an $M^{a_0}_f$-segment (respectively a reverse $^{b_0}M_f$-segment).
Hence,  the equation \eqref{distance to the sublevel}  or \eqref{distance to the sublevel inverse 2} holds for each $q\in V_p.$

$\qedd$
\end{proof}
\begin{remark}
For some $p\in f^{-1}(\inf f,\sup f)\setminus{\cal C}(f)$ 
the neighborhood $V_p$ guaranteed in Lemma \ref{lem3.20} can admit both  points of the singular sets ${\cal C}_ +(f) $ and  ${\cal C}_ -(f),$ even if we choose a smaller one.
In Section 7, we will construct an almost distance function $f$ on Euclidean plane which admits a point in $\overline{ {\cal C}_+(f) }\cap\overline{  {\cal C}_-(f)  }\setminus {\cal C}(f).$
Here $\overline A$ denotes the closure of the set $A.$

\end{remark}

\begin{proposition}\label{corN.6}
Let $f $ be an almost distance function  on a   forward complete and connected Finsler manifold $M,$ 
then, the set 
\begin{equation*}
{\cal C}_+^{(2)}(f):=\{p\in f^{-1}(\inf f,\sup f) |\; \textrm{there exist at least two  $f$-geodesics  to $p$ }    \}
 \end{equation*}
is a subset of ${\cal C}_+(f).$
Moreover, for each point $p\in{\cal C}_+(f)\cap f^{-1}(\inf f,\sup f)$ admitting a unique maximal $f$-geodesic, there exists a sequence of points in ${\cal C}_+^{(2)}(f)$ converging to $p.$
The  subset ${\cal C}_-^{(2)}(f)$ of ${\cal C}_-(f)$  corresponding to ${\cal C}_+^{(2)}(f)$ also   has  the  same  properties for an almost distance function $f$ on a backward complete and connected Finsler manifold.
\end{proposition}
\begin{proof}
Choose any $p\in {\cal C}_+^{(2)}(f).$
Since $p\in\cC_+^{(2)}(f),$ there exist at least two $f$-geodesics  to $p$.
Choose one of them, say
 $\alpha :[f(p)-\delta_1, f(p)]\to M.$  
Suppose that  $p\notin\cC_+(f).$
Thus,
there exists an $f$-geodesic $\gamma :[f(p),f(p)+\delta_2]\to M$ emanating from $p=\gamma(f(p)).$ 
By Lemma \ref{lem: f-geod are smooth at int points},
we get $\dot\alpha(f(p))=\dot\gamma(f(p))$. 
This implies that the point $p$ admits a unique maximal $f$-geodesic to $p,$ a contradiction. 
Hence $p\in {\cal C}_+(f),$ and ${\cal C}_+^{(2)}(f)\subset{\cal C}_+(f).$

Choose any  $p\in \cC_+(f)\cap f^{-1}(\inf f,\sup f)$  admitting a unique maximal  $f$-geodesic.
Suppose that there exists an open  ball $B_\delta(p)\subset f^{-1}(\inf f,\sup f)$  of radius $\delta $ centered at $p$ such that $B_\delta(p)\cap {\cal C}_+^{(2)}(f)=\emptyset.$
From Lemma \ref{lem:f-geod from cut locus}, $p$ has a neighborhood  whose each point
admits a unique maximal $f$-geodesic. 
Thus,
by Theorem A, $f$ is $C^1$ around $p.$
By making use of the same argument in the proof of Proposition 2.5 in \cite{ST}, we get a contradiction.
$\qedd$	
\end{proof}
\begin{remark}
Bishop \cite{Bh} proved Proposition \ref{corN.6} for the distance function from a point on a complete connected Riemannian manifold.  Moreover, it was proved by Sabau \cite{Sa} for Busemann functions on  a forward complete connected Finsler manifold.

\end{remark}


\section{The singular locus of an almost distance function on a 2-dimensional Finsler manifold }

Throughout  this section, $M$ always denotes a connected 2-dimensional Finsler manifold,
unless otherwise stated.
We recall that  a  homeomorphism from  the  closed interval $[0,1]$  into
 $M$ is called a {\it Jordan arc}. 
A topological space $T$ is called a {\it local tree} if 
for any point $x$ in $T$ and any neighborhood $U$ of $x,$
there exists a neighborhood $V\subset U$ of $x$ such that 
 any distinct two points in $V$ can be joined by a  Jordan arc in $V$ which is unique in V.

 A  continuous curve $c\; : \:[a,b]\to M$ 
is called {\it rectifiable} if its length 
\begin{equation}\label{d-length}
l(c):=\sup\{\sum_{i=1}^{k}\;d(c(t_{i-1}),c(t_i)) \: | \: a=:t_0<t_1<\dots<t_{k-1}<t_k:=b\}.
\end{equation}
is finite. 
Remark that  it is known that $l(c)$ equals $\int_a^b F(\dot c(t))dt$ for a Lipschitz curve $c: [a,b]\to M$ (for example, see \cite[Theorem 7.5] {ST}).

The {\it intrinsic metric} $\delta$ on $\cC(f)\cap f^{-1}(\inf f,\sup f)$ is defined as:
\begin{equation*}
\delta(q_1,q_2):=
\begin{cases}
\inf\{l(c)|\ c\ \textrm{is a rectifiable arc in }\cC(f)\cap f^{-1}(\inf f,\sup f)\ \textrm{joining } q_1\ \textrm{and } q_2\},\\
 \qquad \qquad \qquad \qquad\textrm{if $q_1,q_2\in \cC(f)$ are in the same connected component,}\\
+\infty, \qquad\qquad \qquad\textrm{otherwise}.
\end{cases}
\end{equation*}\bigskip

{\bf  Proof of Theorem B}\\
From Theorem \ref{thm:cut locus and distance properties}  (respectively Theorem \ref{thm:cut locus and distance properties 2})
it follows that
for each point $p\in\cC_+(f)\cap f^{-1}(\inf f,\sup f),$ (respectively $p\in\cC_-(f)\cap f^{-1}(\inf f,\sup f),$)
there exist a neighborhood $V_p\subset f^{-1}(\inf f , \sup f)$ of $p$ and a sublevel set $M^{a_0}_f$ (respectively a superlevel set $^{b_0}M_f$)  of $f$
such that $\cC_+(f)\cap V_p=\cC(M^{a_0}_f)\cap V_p$ (respectively $\cC_-(f)\cap V_p=\cC_{rev}(^{b_0}M_f)\cap V_p$).
Since $M$ is separable, $\cC(f)\cap f^{-1}(\inf f,\sup f)$ is covered by the union of countably many open sets $V_{p_i},$ $p_i\in \cC(f)\cap f^{-1}(\inf f,\sup f).$
By applying Theorem B in \cite{ST} to the sublevel sets $M_f^{a_0}$ and the superlevel sets $^{b_0}M_f$ determined from each point $p_i$ above,
  we obtain Theorem B. Note that from Theorem 6.4 in \cite{ST} it follows that  every Jordan arc in $\cC(f)\cap f^{-1}(\inf f,\sup f)$ is rectifible.
\bigskip

We need  the inverse function theorem for a Lipschitz map which was proved by F.C. Clarke \cite{CF1} to prove Theorem  C and its corollary.
The Lipschitz  version of the inverse function theorem is a tool of nonsmooth analysis developed  by him \cite{CF2, CLSW}.

Let $\phi : U\to \R^n$ denote a locally Lipschitz map from an open subset $U$ of 
$\R^n$ into $\R^n.$
The {\it generalized differential} $\partial \phi_p$ at a point $p \in U$
is defined by
\begin{equation*}
 \partial \phi_p:={\rm co}\{ \lim_{i\to \infty}d\phi_{p_i} | \; \{p_i\} \textrm{ converges to $p$ and $d\phi_{p_i}$ exists for each $p_i$}  \},
\end{equation*}
where ${\rm co}(A)$  denotes the convex hul of the set $A,$ when $A$ is a subset of a linear space.
Note that from  Rademacher's theorem the differential $d\phi$ of the local Lipschitz map $\phi$ exists almost everywhere.

\begin{definition}
A point $p\in U$ is called {\it nonsingular} if each element of  $\partial\phi_p$
 is of maximal rank, otherwise, it is called {\it singular}.
\end{definition}

The following theorem was proved by F. H.  Clarke \cite{CF1}.
\begin{theorem}
Let $\phi :U\to \R^n$ be a Lipschitz map from an open subset $U$ of $\R^n$ into $\R^n.$
If a point $p\in U$ is nonsigular for $\phi,$ then there exist neighborhoods $U_p, V_{\phi(p)}$  of $p$ and $\phi(p)$ respectively such that
$\phi|_{U_p} $ is a bi-Lipschitz homeomorphism  from $U_p$ onto $V_{\phi(p)}.$ 
\end{theorem}
As a corollary to this theorem, we get the implicit function theorem for a Lipschitz function. 
\begin{theorem}\label{th6.4}
Let $f$ be a Lipschitz function defined on a open subset $U$ of $\R^n.$
If  a point $p\in U$ is nonsigular for the function $f,$
then there exists an open neighborhood  $U_p\subset U$  of $p$ such that $U_p\cap f^{-1}(f(p))$ is a topological hypersurface which is bi-Lipschitz homeomorphic to an open subset of $\R^{n-1}.$
\end{theorem}

 The generalized differential is naturally defined for a  locally Lipschitz map between smooth manifolds (see \cite{KTsp}). 
Some tools of nonsmooth analysis are introduced in differential geometry and used in  the proof of differentiable sphere theorems (see \cite{KTsp}).

Now, let us return to our 
situation.
Let $f$ denote an almost distance function on a connected Finsler manifold $M.$
By definition, a  point $p\in M$  is  singular for $f$ if and only 
if $\partial f_p$ has zero.
A singular point of the almost distance function $f$ is called {\it critical} in the sense of Clarke.

\begin{definition}
Let $C_r(f)$ denote the set of all critical points of $f$ in the sense of Clarke and $C_V(f):=f(C_r(f)),$ 
each element of which  is called  a {\it critical value} of $f.$ 
\end{definition}

Let us determine explicitly the generalized differential $\partial f_p$ at a  point $p.$
From Theorem A, it follows that $df_q(\cdot)=g_{\dot\gamma(f(q))}(\dot\gamma(f(q)),\cdot) $ for a differentiable point $q,$ where $\gamma$ denotes the unique $f$-geodesic through $q$ with canonical parameter.
Thus,
we get


\begin{equation}
\partial f_p={\rm co}\{ \omega_p(\gamma)| \; \gamma \textrm{ is an }  f\textrm{-geodesic through }  p  \},
\end{equation}
where $\omega_p(\gamma):=g_{\dot\gamma{(f(p)) } }(\dot\gamma(f(p)),\cdot ).$

A linear combination  $\sum_{i=1}^k\lambda_i \omega_p(\gamma_i)$
of $\omega_p(\gamma_i), 1\leq i\leq k,$ where  each $\gamma_i$ denotes an $f$-geodesic through $p,$ 
is an element of $\partial f_p,$ if $\sum_{i=1}^k\lambda_i=1$ and $\lambda_i\ge 0.$
Moreover,  the set of all such  linear combinations of $\omega_p(\gamma),$ where $\gamma$ denotes an $f$-geodesic,  is convex.
Therefore, we obtain,

\begin{lemma}\label{lem6.5}
The generalized differential $\partial f_p$ of $f$ at $p$ is given by 
$$\partial f_p=\{ 
\sum_{i=1}^k\lambda_i\omega_p(\gamma_i) |\; \sum_{i=1}^k\lambda_i=1, \lambda_i\geq 0,
\textrm{each } \gamma_i \textrm{ is an }f\textrm{-geodesic through } p 
 \}.$$
\end{lemma}


\begin{lemma}\label{lem6.6}
Let $c:(a,b)\to M$ be a continuous   curve on a connected and bi-complete Finsler manifold $M,$ which 
is differentiable at some
 $t_0\in(a,b),$  and $f$ an almost distance function on the manifold $M.$
If $f\circ c$ is also differentiable at $t_0,$ and if $p:=c(t_0)\in\cC(f)\cap f^{-1}(\inf f,\sup f),$  then
$(f\circ c) '(t_0)= \omega_p(\gamma)(\dot c(t_0) )$ holds for any $f$-geodesic $\gamma$  through $p.$
In particular, $(f\circ c) '(t_0)= 0$ if $p$ is a critical point of $f$ in the sense of Clarke.
\end{lemma}
\begin{proof}
Without loss of generality, we may assume that $t_0=0.$
Since $p\in\cC(f),$  it is clear that  $p\in\cC_+(f)$ or $p\in\cC_-(f).$
We will prove our lemma by assuming $p\in\cC_+(f)$.  The other case can be similarly proved.

Choose any $f$-geodesic $\gamma$ through $p$ with canonical parameter. By Lemma \ref{lem:f-geod from cut locus}, $\gamma$ is defined on $[f(p)-\delta(p),f(p)]$ and $\gamma(f(p))=p$ holds.
Since $f$ is 1-Lipschitz and $\gamma$ is an $f$-geodesic,
$$f(c(s))-f(p)\le d(\gamma(f(p)-\delta(p)),c(s))-d(\gamma(f(p)-\delta(p)),p)$$
holds for any $s\in(a,b).$
Hence for any $s<0<t$ in $(a,b)$
the equations
$$\frac{f(c(s))-f(p)}{s}\geq\frac{d(\gamma (f(p)-\delta(p)),c(s))-d(\gamma(f(p)-\delta(p)),p)}{s}$$
and 
$$\frac{f(c(t))-f(p)}{t}\leq\frac{d(\gamma (f(p)-\delta(p)),c(t))-d(\gamma(f(p)-\delta(p)),p)}{t}$$
hold.
By taking the limits of  the above equations with respect to $s,t$ respectively and  by \eqref{first variation},
we get
$$\omega_p(\gamma)(\dot c(0))\geq (f\circ c)'(0)\geq \omega_p(\gamma)(\dot c(0)).$$
Therefore, $(f\circ c)'(0)=\omega_p(\gamma)(\dot c(0))$ holds for any $f$-geodesic $\gamma$ through $p.$

Suppose that the point $p$ is a critical point of $f.$ 
It follows from Lemma \ref{lem6.5} that  for any $\omega_p\in\partial f_p,$ we have $(f\circ c)'(t_0)=\omega_p(\dot c(t_0)).$ 
From our assumption, $\partial f_p$ contains the zero 1-form. Thus $(f\circ c)'(t_0)=0.$

$\qedd$
\end{proof}

We need the following lemma to prove Lemma \ref{lem6.8}, which is the Sard Theorem for a continuous function. The proof is given in \cite[Lemma 3.2]{ShT}.

\begin{lemma}\label{lem6.7}
Let $h :(a,b)\to R$ be a continuous function.
Then 
the set
$h(D_0(h))$ is of (Lebesgue) measure zero,
where
$D_0(h):=\{ t\in(a,b)\: | \; h'(t) \textrm { exists and equals } 0\}.$

\end{lemma}
\begin{lemma}\label{lem6.8}
Let $c:(a,b)\to \cC(f)\cap f^{-1}(\inf f,\sup f)$ be a unit speed Lipschitz curve.
Then the set
$(f\circ c)(C_r(c) )$ is of measure zero,
where $$C_r(c):=\{t\in(a,b) \; |\; c(t) \in C_r(f)\}.$$

\end{lemma}

\begin{proof}

Let $ND( c)$ denote the set  of all $t\in(a,b)$ at which $c$ is not differentiable
and 
$ND(f\circ c)$ the set of  all $t\in(a,b)$ at which $f\circ c$ is not differentiable.

By  Rademacher's theorem, it follows that both  sets $ND(c)$ and $ND(f\circ c)$ are  of measure zero.
Choose any $t\in(a,b)\setminus\left( ND( c) \cup ND(f\circ c)  \right).$
Hence $\dot c(t)$ and $(f\circ c)'(t)$ exist. If $c(t)$ is a critical point of $f$ in the sense of Clarke, then by Lemma \ref{lem6.6}, $(f\circ c)'(t)=0.$ Thus, the set $C_r(c)$ is a subset of 
$ND(c)\cup ND(f\circ c)\cup D_0(f\circ c),$
where $D_0(f\circ c)=\{ t\in(a,b) \; | \: (f\circ c)'(t) \textrm{  exists and equals }0 \}.$
Since $f\circ c$ is a Lipschitz function and $ND(c)\cup ND(f\circ c)$ is of measure zero, its  image by $f\circ c$ is also of measure zero.
Therefore, by Lemma \ref{lem6.7}, $(f\circ c)(C_r(c))$ is of measure zero.

$\qedd$
\end{proof}

{\bf  Proof of Theorem C}\\

From Theorem B, it follows that
there exist a countably many unit speed Lipschitz curves $m_i :[a_i,b_i]\to \cC(f)\cap f^{-1}(\inf f,\sup f)$ such that 
$\cC(f)\cap f^{-1}(\inf f,\sup f)\setminus E=\bigcup_{i=1}^\infty m_i[a_i,b_i],$
where $E$ denotes the set of all end points of $\cC(f).$
Let $E^{(2)}\subset E$ denote the  set of end points admitting more than one $f$-geodesic.

It follows from the proof of \cite[Theorem A(3)]{ShT}, Theorems \ref{thm:cut locus and distance properties} and \ref{thm:cut locus and distance properties 2} 
that  the set $E^{(2)}$ is a countable  set.
It is clear that 
$C_V(f)$ is a subset of the union of    
$\{\inf f,\sup f\},$ $\bigcup_{i=1}^\infty\left( (f\circ m_i)(\widetilde C_r( m_i) ) ) \right)$ and $ f(E^{(2) }),$
where $\widetilde C_r(m_i):=C_r(m_i)\cup\{a_i,b_i \}.$
Hence, by Lemma \ref{lem6.8},  the set $C_V(f)$ is of measure zero.
$\qedd$

\bigskip

Now,   {\bf Corollary to Theorem C} is clear from Theorem \ref{th6.4}. 

\begin{remark}
By Rademacher's theorem it was proved  that $ND(m_i)$ and $ND(f\circ m_i)$ are of measure zero.  Furthermore one can conclude that
for each $m_i :[a_i,b_i]\to \cC(f)\cap f^{-1}(\inf f,\sup f),$ 
$ND(m_i)$ and $ND(f\circ m_i)$ are countable. Indeed, it follows from  the proofs of  \cite[Theorem A(3)]{ShT} and \cite[Lemma 9.1]{ST} that there exist at most countably many points
on the curve $m_i$  admitting more than two $f$-geodesics, and from \cite[Propositions 2.1 and 2.2 ]{ST} it follows that $(f\circ m_i)'(t)$ and $\dot m_i(t)$ exist if $m_i(t)$ admits exactly two $f$-geodesics.
This property is very close to \cite[Corollary 10]{T}, which says that
the distance function to the cut locus of a closed submanifold $N$ of a complete 2-dimensional Riemannian manifold is differentiable except for a countably many points in the unit normal bundle of $N.$
\end{remark}
\begin{remark}
Theorem C is still true for an  arbitrary dimensional Riemannian manifold, if the function $f$ is smooth ($C^\infty$) on $f^{-1}(\inf f,\sup f)\setminus\cC(f)$, and 
$\cC(f)$ is closed. 
In fact, by combining Theorems \ref{thm:cut locus and distance properties},
\ref{thm:cut locus and distance properties 2}   and \cite[Theorem 1]{R},
we get:
\begin{theorem}\label{remThC}
Let $f$ be an almost distance function on a complete arbitrary dimensional Riemannian  manifold $M.$
If the function $f$ is smooth on $f^{-1}(\inf f,\sup f)\setminus \cC(f)$　and $\cC(f)$ is closed,
then the set $C_V(f)$ of the critical values of $f$  is of measure zero.
\end{theorem}

\end{remark}


\section{Examples  of almost distance functions}
In this section, we construct  almost distance functions $f$ on Euclidean plane which admits  a point in $\left(f^{-1}(\inf f,\sup f)\setminus{\cal C}(f) \right)\cap \overline{ {\cal C}_+(f)   }\cap\overline{{\cal C}_-(f)}.$
Here $\overline A$ denotes the closure of the set $A.$
Let $\R^2$ denote Euclidean plane with canonical coordinates $(x,y)$ with the origin $o=(0,0).$
Let $D_1$ denote the unit closed ball centered at the origin, so that $D_1=\{(x,y)| x^2+y^2\leq1\}.$ 
Choose any strictly decreasing sequence $\{\theta_i\}_{i=1}^\infty$ with $\theta_1<\pi$ convergent to zero.
For each $i\geq 1,$
put $p_i:=(2\cos\omega_i,2\sin\omega_i),$ where $\omega_i:=(\theta_i+\theta_{i+1})/2.$
It is clear that for each $i\geq 1,$ both points $(\cos\theta_i,\sin\theta_i),$ and 
$ (\cos\theta_{i+1},\sin\theta_{i+1})$ lie on the common circle centered at $p_i$ with radius $r_i:=d(p_i,(\cos\theta_i,\sin\theta_i) ).$
We define a closed subset $N$ of $\R^2$ by
$$N:=D_1\setminus\bigcup_{i=1}^\infty B_{r_i}(p_i),$$
where  $B_{r_i}(p_i)$ denotes the open ball centered at $p_i$ with radius $r_i.$
The function $d_N$ is an almost distance function and $p_i\in{\cal C} _+(d_N)$ for each $i.$
Thus, the point $(2,0)$ is in the closure of ${\cal C}_+(d_N),$ but it is an  interior point  of the maximal $d_N$-geodesic, $\{(t,0)\; |\:  t\geq1\}.$ 

It is clear to see that the rays 
$R_\theta(o):=\{ (r\cos\theta,r\sin\theta)\; | \: r\geq1 \}, \theta\in(\theta_1,2\pi)\cup\{\theta_i\; | \: i\geq 1   \}$ 
emanating from $N$ 
are maximal $N$-segments, and hence maximal $d_N$-segments.
Note that $d_N(t,0)=|t|-1$ for any $t$ with $|t|\geq1.$ 

Next we will construct the function $\eta$ defined by a  sequence of closed subsets $\{C_n\}_{n=3}^\infty$ (see Example \ref{EWu almost dist} and \cite{Wu}). 
For each $\theta_i,$ and $n\geq3,$  put $q_i^{(n)}:=(n\cos\theta_i,-n\sin\theta_i).$
Recall that $\{\theta_i\}$ denote the strictly decreasing  sequence convergent to zero.
Hence, for each $i$ and  $n\geq3,$ the point $u_i:=(2\cos\omega_i,-2\sin\omega_i),$ where $\omega_i=(\theta_i+\theta_{i+1})/2$ is equidistant from $q^{(n)}_i$ and $q^{(n)}_{i+1}.$
For each $n\geq3,$ we define a closed subset $C_n$ by
$$C_n:=  \R^2\setminus\bigcup_{i=1}^\infty \left(B_{r_i^{(n)}}(u_i)\cup B_n(o) \right),  $$
where $r_i^{(n)}:=d(u_i,q_i^{(n)}).$
We define an almost distance function $\eta$ on $\R^2$ by
$$\eta(p):=\lim_{n\to\infty} (n-d(p,C_n) ).$$
 It is clear to see that 
the rays $R_\theta(o):=\{(r\cos\theta,r\sin\theta)|r\geq0\}, \theta\in[0,2\pi-\theta_1]\cup\{ 2\pi-\theta_i\; |\: i\geq 2\}$ emanating from $o$ are $\eta$-geodesics.
For each $u_i,$ the rays $R_\theta(u_i):=\{  (r\cos\theta,r\sin\theta)+u_i\; |\: \: r\geq 0  \}, \theta\in[2\pi-\theta_i,2\pi-\theta_{i+1}]$ emanating from $u_i$ are $\eta$-geodesics.
For each point $p$ on the line segment $ou_i$ joining  $o$ to $u_i,$ 
the two rays
$R_\theta(p):=\{ (r\cos\theta,r\sin\theta)+p\; |\:  r\geq0 \}, \theta=2\pi-\theta_i,2\pi-\theta_{i+1}$ are 
$\eta$-geodesics.
Hence $u_i\in{\cal C}_-(\eta)$ for each $u_i$ and each line segment $ou_i$ is  a subset of ${\cal C}_-(\eta).$ In particular, the point $(2,0)$ is in the closure of ${\cal C}_-(\eta),$ since $\lim_{i\to\infty}u_i=(2,0).$

Now we will  construct another almost distance function $f_\eta^N$ by combining $\eta$ and $d_N.$
Since $\eta(x,0)=|x|$ for all $x\in{\R},$ $d_N(x,0)=0$ for all $x$ with $|x|\leq 1,$ and $\eta(x,0)=d_N(x,0)+1=|x|$ for all $x$ with $|x|\geq1,$
$\eta_1(x,0)=d_N(x,0)+1$ holds for all real number $x,$
where $\eta_1$ denotes  a 1-Lipschitz function on $\R^2$ defined by $\eta_1(x,y):=\max\{\eta(x,y),1\}.$
Thus we may define a 1-Lipschitz function $f_\eta^N$ on $\R^2$
by
$f_\eta^N(x,y)=d_N(x,y) +1$ for $y\geq0,$ and
$ f_\eta^N(x,y)=\eta_1(x,y)$ for $y\leq0.$
Note that $\inf f_\eta^N=1$ and $\sup f_\eta^N=\infty.$
It is easy to check that
the function $f_\eta^N$ is an almost distance function on $\R^2$ 
and that 
the point $(2,0)$ is in the closure of ${\cal C}_-(f_\eta^N)$ and the closure of ${\cal C}_+(f_\eta^N).$

By imitating the way above,
it is possible to construct an almost distance function $f$ which admits infinitely
many points in $(\overline{ {\cal C} _+(f) } \cap \overline{ {\cal C}_-(f)    })\setminus{\cal C}(f) .$ 
Indeed, for each pair of positive numbers $a<b<\pi/2,$ we proved that there exists an
almost distance function $f_{ab}$ on $\R^2$ such that $(2\cos\omega,2\sin\omega)$ is in $(\overline{ {\cal C} _+(f_{ab}) } \cap \overline{ {\cal C}_-(f_{ab})    })\setminus{\cal C}(f_{ab}) $ and each ray  $R_\theta(o):=\{(r\cos\theta,r\sin\theta)\; |\: r\geq1\}$
emanating from the unit circle $x^2+y^2=1$ is an $f_{ab}$-geodesic if $\theta\in[0,2\pi]\setminus(a,b).$
	Thus, if  a strictly decreasing sequence $\{\epsilon_n\}_{n=1}^\infty$ with $\epsilon_1<\pi/2$ convergent to 0 is given,
we get   a sequence of  almost distance functions
$f_n:=f_{\epsilon_n\epsilon_{n+1}}$ on $\R^2$
such that the point $(2\cos\omega_n,2\sin\omega_n),$  where $\omega_n:=(\epsilon_n+\epsilon_{n+1})/2,$ is in $(\overline{ {\cal C} _+(f_n) } \cap \overline{ {\cal C}_-(f_n)    })\setminus{\cal C}(f_n) $
and  each ray $R_\theta(o):=\{(r\cos\theta,r\sin\theta)\; |\: r\geq1\}$
emanating from the unit circle $x^2+y^2=1$ is an $f_{n}$-geodesic if $\theta\in[0,2\pi]\setminus(\epsilon_{n+1},\epsilon_n).$
Therefore, it is easy to construct an almost distance function $f$ on $\R^2$ such that
for each $n$ the point $(2\cos\omega_n,2\sin\omega_n)$ is an element of 
$(\overline{ {\cal C} _+(f) } \cap \overline{ {\cal C}_-(f)    })\setminus{\cal C}(f) $
and each ray $R_\theta(o):=\{(r\cos\theta,r\sin\theta)\; |\: r\geq1\}$
emanating from the unit circle $x^2+y^2=1$ is an $f$-geodesic if
 $\theta\in[0,2\pi]\setminus(0,\epsilon_1).$
\bigskip

\noindent
{\bf Acknowledgments}\medskip\\
The author  would like to thank Professor Sorin V. Sabau for his valuable comments on Finsler geometry.




\bigskip

\noindent School of Science,
Department of Mathematics,\\
Tokai University,
Hiratsuka City, Kanagawa Pref., 259\,--\,1292,
Japan

\medskip
\noindent
{\tt
tanaka@tokai-u.jp}


\begin{thebibliography}{MMMM}
	
	\bibitem[AZ]{AZ}
 H. Akbar-Zadeh, {\it Sur les espaces de Finsler \`a courbures sectionnelles constants}, Acad. Roy. Belg. Bull. Cl. Sci (5) {\bf 74} (1988), 281--322.
	
	\bibitem[BCS]{BCS}
	D.~Bao, S. S.~Chern, Z.~Shen,
	An Introduction to Riemann Finsler Geometry, Springer, GTM
	\textbf{200},
	2000.
	
	\bibitem[Bh]{Bh}
	R. L. Bishop, {\it Decomposition of cut loci,}
	Proc. Amer. Math. Soc. {\bf 65 } (1977),  133--136.
	
	\bibitem[Bu1]{Bu1}
	H. Busemann, {\it Local metric geometry,} Trans. Amer. Math. Soc. {\bf 56}
	(1944), 200--274.
	
	\bibitem[Bu2]{Bu2}
	H. Busemann, { The geometry of Geodesics}, New York: Acad.
	Press, 1955.
	
\bibitem[CF1]{CF1} F. H. Clarke, {\it On the inverse function theorem,}
Pacific J.  Math. Soc. {\bf 64} (1976) 97--102.
\bibitem[CF2]{CF2} F. H. Clarke, {Optimization and nonsmooth analysis, }
Classics in applied mathematics ; 5, SIAM, Philadelphia, 1990
\bibitem[CLSW]{CLSW} F. H. Clarke, Yu. S. Ledyaev, R. J. Stern and P. R. Wolensky,
 { Nonsmooth analysis and optimal theory,} Graduate texts in mathematics; 178,  Springer-Verlag, New York-Berlin-Heidelberg, 1998.
\bibitem[Cu]{Cu}
X. Cui, {\it Viscosity solutions, ends and boundaries,} Illinois J. Math. {\bf 60}
(2016), 459--480.


\bibitem
[Fe]{Fe}
 S. Ferry,
{\it When $\epsilon$-boundaries are manifolds,}
Fund. Math. {\bf 90} (1975/76),  199–-210.

\bibitem[Fu]{Fu}
J. H. G. Fu, {\it Tubular neighborhoods in Euclidean spaces,} Duke Math. J., {\bf 52} (1985), 1025--1046.

	
	\bibitem[In1]{In1}
	N. Innami, {\it Differentiability of Busemann Functions and
		Total Excess},
	Math. Z. {\bf 180} (1982), 235--247.
	
	\bibitem[In2]{In2}
	N. Innami, {\it On the terminal points of co-rays and rays},
	Arch. Math., {\bf 45} (1985), 468--470.
\bibitem[IT]{IT}
 J. Itoh, and  M. Tanaka, {\it  A  Sard theorem for the distance function,} Math. Ann. {\bf 320} (2001), 1–-10.	


	\bibitem[KT1]{KT1}
	K.~Kondo, M.~Tanaka, {\it Total curvatures of model surfaces
		control topology of complete open manifolds with radial
		curvature bounded below: I ,}
	Math. Ann. {\bf 351} (2011), 251--256.
	
	\bibitem[KT2]{KTsp} K. Kondo, M. Tanaka, {\it Approximations of Lipschitz maps via immersions and differentiable exotic sphere theorems,}
Nonlinear Anal. {\bf 155} (2017), 219–-249.

\bibitem[MM] {MM}
 C. Mantegazza and  A. C. Mennucci,
{\it Hamilton-Jacobi equations and distance functions on Riemannian manifolds,} 
Appl. Math. Optim. {\bf 47} (2003), 1–25. 

	\bibitem[M]{M}
	F. Morgan, { Geometric measure theory, A Beginner's Guide}, Academic Press, 1988.
	
	\bibitem[N1]{N1}
Y. Nasu, {\it On asymptotic conjugate points,} Tohoku Math.J.(2), {\bf 7} (1955), 157--165.
	
	\bibitem[N2]{N2}
	Y. Nasu, {\it On asymptotes in a metric space with non-positive curvature}, Tohoku Math. J.,(2) {\bf 9} (1957), 68--95.
	
	
	
	
	\bibitem[Oh]{Oh}
	S. Ohta, {\it Splitting theorems for Finsler manifolds of
		nonnegative Ricci curvature,} J. Reine Angew. Math. {\bf 700} (2015), 155–--174.
	
	
\bibitem[P]{P}

H. Poincar\'e, 
{Sur les lignes g\'eod\'esiques des surfaces convexes,}
Trans. Amer. Math. Soc. {\bf 6} (1905), 237–-274



\bibitem[R]{R}
L. Rifford,
{\it A Morse-Sard theorem for the distance function on Riemannian manifolds,}
Manuscripta Math. {\bf 113 } (2004), 251–-265.




		\bibitem[Sa]{Sa}
		S. V. Sabau,
		{\it The co-points of rays are cut points of upper level sets for Busemann functions}, SIGMA
Symmetry Integrability Geom. Methods Appl.  12,  (2016) Paper No.036, 12 pages.

	\bibitem[S]{S}
	Z.~Shen,
	{\it Lectures on Finsler Geometry}, World Scientific, 2001.
	
	
	\bibitem[Sh]{Sh}
	K.~Shiohama,
	{\it Topology of Complete Noncompact Manifolds},
	Geometry of Geodesics and related Topics, ASPM {\bf 3}, 1984,
	423--450.
	
	\bibitem[ShT]{ShT}
K. Shiohama and M. Tanaka, {\it Cut loci and distance spheres on Alexandrov surfaces,} Actes de la Table Ronde de  G\'eom\'etrie  Diff\'erentielle (Luminy, 1992), 531--559, S\'emin. Congr., 1, Soc. Math. France, Paris, 1996. 
	
	
	\bibitem[SST]{SST}
	K.~Shiohama, T.~Shioya, and M.~Tanaka,
	The Geometry of Total Curvature on Complete Open Surfaces,
	Cambridge tracts in mathematics \textbf{159},
	Cambridge University Press, Cambridge, 2003.
	
	
	\bibitem[ST]{ST}
	 S. V. Sabau, M. Tanaka, {\it The cut locus and distance function
		from a closed subset of a Finsler manifold}, Houston J. of Math., {\bf 42}   (2016), 1157-1197.
	
\bibitem[T]{T}

M. Tanaka, {\it  Characterization of a differentiable point of the distance function to the cut locus,} J. Math., Soc., Japan, {\bf 55} (2003), 231--241.
		\bibitem[WZ]{WZ}
	          R. L. Wheeden, A. Zygmund, Measure and Integral, Marcel Decker, New York, 1977.

                \bibitem[W]{W}
                  J. H. C. Whitehead, {\it On the covering of a complete space by the geodesics through a point}, Ann. of Math., {\bf 36} (1935), 679--704.
	        \bibitem[Wu]{Wu}
                  H. Wu, {\it An elementary method in the study of nonnegative curvature}, Acta Math.
{\bf 142} (1979), 57--78.
	
\end{thebibliography}
\end{document}